\documentclass[11pt,letterpaper]{amsart}
\usepackage{amsmath,amsthm,amssymb,mathrsfs}
\usepackage[dvipsnames]{xcolor}
\usepackage[margin=1.25in]{geometry}
\usepackage[shortlabels]{enumitem}
\usepackage{tikz,tikz-cd}
\usetikzlibrary{positioning,arrows.meta}
\usepackage{setspace}
\usepackage{etoolbox}
\usepackage{hyperref}
\usepackage{comment}
\usepackage{float}
\IfFileExists{needspace.sty}{\usepackage{needspace}}{\newcommand{\needspace}[1]{\par}}
\definecolor{edgeRed}{RGB}{214,39,40}
\definecolor{guideBlue}{RGB}{40,58,145}
\def\opp{\mathrm{opp}}
\def\ge{\geqslant}
\def\le{\leqslant}
\def\a{\alpha}
\def\b{\beta}
\def\g{\gamma}
\def\G{\Gamma}
\def\d{\delta}
\def\D{\Delta}
\def\L{\Lambda}
\def\e{\epsilon}

\def\o{\omega}

\def\th{\theta}

\def\l{\lambda}

\def\aff{\mathrm{aff}}

\def\QBG{\mathrm{QBG}}

\def\TBG{\mathrm{TBG}}
\def\tilt{\mathrm{tilt}}

\def\<{\langle}
\def\>{\rangle}

\newcommand{\fkh}{\ensuremath{\mathfrak{h}} }

\newcommand{\fkt}{\ensuremath{\mathfrak{t}} }

\newcommand{\bq}{\mathbf q}

\newcommand{\BC}{\ensuremath{\mathbb {C}} }

\newcommand{\BF}{\ensuremath{\mathbb {F}} }
\newcommand{{\BG}}{\ensuremath{\mathbb {G}} }

\newcommand{{\BK}}{\ensuremath{\mathbb {K}} }

\newcommand{\BQ}{\ensuremath{\mathbb {Q}} }
\newcommand{\BR}{\ensuremath{\mathbb {R}} }
\newcommand{\BS}{\ensuremath{\mathbb {S}} }

\newcommand{\BZ}{\ensuremath{\mathbb {Z}} }

\newcommand{\CB}{\ensuremath{\mathcal {B}} }

\newcommand{\CD}{\ensuremath{\mathcal {D}} }

\newcommand{\CI}{\ensuremath{\mathcal {I}} }

\newcommand{\CT}{\ensuremath{\mathcal {T}} }
\newcommand{\CU}{\ensuremath{\mathcal {U}} }

\newcommand{\ud}{\underline}

\newcommand{\GL}{\mathrm{GL}}

\DeclareMathOperator{\Lie}{Lie}

\DeclareMathOperator{\Spec}{Spec}

\newcommand{\wt}{\mathrm{wt}}

\def\ta{\tilde \a}

\def\tg{\tilde \g}

\newtheorem{theorem}{Theorem}

\newtheorem{proposition}[theorem]{Proposition}
\newtheorem{lemma}[theorem]{Lemma}
\newtheorem {conjecture}[theorem]{Conjecture}

\theoremstyle{definition}
\newtheorem{definition}[theorem]{Definition}
\newtheorem{example}[theorem]{Example}
\newtheorem*{example*}{Example}
\newtheorem{thmintro}{Theorem}

\newtheorem{conjintro}[thmintro]{Conjecture}

\newtheorem{remark}[theorem]{Remark}

\newtheorem*{function*}{Function}

\numberwithin{equation}{section}
\numberwithin{theorem}{section}

\setitemize[0]{leftmargin=*,itemsep=\the\smallskipamount}
\setenumerate[0]{leftmargin=*,itemsep=\the\smallskipamount}

\renewcommand{\to}{%
   \ifbool{@display}{\longrightarrow}{\rightarrow}%
   }
\let\shortmapsto\mapsto
\renewcommand{\mapsto}{%
   \ifbool{@display}{\longmapsto}{\shortmapsto}%
   }

\newif\ifshowrevisions
\showrevisionstrue

\hypersetup{hidelinks,pdftitle={Affine Realization of Tilted Richardson Varieties -- review copy},pdfauthor={Xuhua He and Qingchao Yu}}
\begin{document}

\title[]{Affine realization of tilted Richardson varieties} 



\thanks{}

\author[Xuhua He]{Xuhua He}
\address{Department of Mathematics and New Cornerstone Science Laboratory, The University of Hong Kong, Pokfulam, Hong Kong, Hong Kong SAR, China}
\email{xuhuahe@hku.hk}

\author[Qingchao Yu]{Qingchao Yu}
\address{Institute for Advanced Study, Shenzhen University, Nanshan District, Shenzhen, Guangdong, China}
\email{qingchao\_yu@outlook.com}

\begin{abstract}
Richardson varieties provide geometric representatives for products of Schubert classes, with stratifications indexed by the subintervals of Bruhat intervals. We construct analogous geometric models for tilted Bruhat intervals for all split connected reductive groups. Our construction uses Richardson varieties in affine flag varieties and yields abstract tilted Richardson varieties whose isomorphism classes, together with their labeled stratifications, are independent of the choices involved.

We obtain an explicit formula for the point-counting polynomials of their open parts in terms of increasing minimal weight paths, extending Dyer's formula for Kazhdan--Lusztig $R$-polynomials. 
Over the real numbers, their totally nonnegative parts are regular CW complexes homeomorphic to closed balls, which gives a natural geometric realization of the posets of tilted Bruhat subintervals in all Lie types. 
Over the complex numbers, we prove that their fundamental classes push forward under a continuous loop evaluation map to the Poincar\'e duals of the minimal-degree coefficients of quantum Schubert products.

\end{abstract}

\maketitle

\section*{Introduction}\label{sec:introduction}

\subsection{Richardson varieties and quantum Schubert calculus}
Schubert calculus studies intersections of Schubert varieties in the flag variety $G/B$ of a connected complex reductive group $G$. Richardson varieties arise by intersecting a Schubert variety with an opposite Schubert variety. Each nonempty Richardson variety corresponds to a Bruhat interval and decomposes into open Richardson varieties indexed by its subintervals. The closure relations among these strata are precisely the inclusion relations among the corresponding subintervals. Their cohomology classes represent products of Schubert classes. 


Quantum Schubert calculus extends this intersection theory by
incorporating counts of rational curves meeting Schubert varieties.
The quantum Chevalley formula is encoded by the quantum Bruhat graph,
whose shortest directed paths define the tilted Bruhat intervals of
Brenti--Fomin--Postnikov \cite{BFP}. These intervals are defined for
every pair $u,v\in W_0$ and recover ordinary Bruhat intervals when
$u\leq v$. This leads to a natural geometric problem: construct
analogues of Richardson varieties whose stratifications realize
tilted Bruhat intervals and whose geometry reflects quantum
Schubert multiplication.

In type $A$, Gao--Gao--Gao constructed tilted Richardson varieties
inside the complete flag variety \cite{GGG26}. They established
their irreducibility and stratification, introduced tilted
$R$-polynomials and totally nonnegative parts, and identified
the varieties with minimal-degree two-point curve neighborhoods.
Their construction uses features specific to type $A$.

\subsection{Abstract tilted Richardson varieties}

In this paper, we construct abstract tilted Richardson varieties
uniformly for all split connected reductive groups. Our principal
geometric result is that their isomorphism classes, together with
all labeled strata, are independent of the choices in the affine
construction. Over the complex numbers, we also identify the
pushforwards of their fundamental classes with the Poincar\'e duals
of the minimal-degree coefficients of quantum Schubert products.
These results connect the geometry of the models with both the
interval combinatorics and the quantum intersection theory that
motivate them.

Our approach draws on the relationship between quantum and affine
Schubert calculus developed by Peterson and Lam--Shimozono
\cite{LS10}. Building on the relation between the quantum Bruhat
graph and affine Bruhat order, we realize tilted Bruhat intervals
as affine Bruhat intervals and take the corresponding affine
Richardson varieties as geometric models. The essential geometric
step is to compare the models arising from different choices.
We construct isomorphisms compatible with every stratum labeled
by a tilted Bruhat subinterval. This gives the following result.

\begin{thmintro}
\label{thm:intro-construction}
For every pair $u,v\in W_0$, the affine construction gives a
projective irreducible variety with the following properties.
\begin{enumerate}
\item Its dimension is the length of a shortest directed path
from $u$ to $v$ in the quantum Bruhat graph, and its strata are
indexed by the subintervals of $[u,v]_{\mathrm{tilt}}$.

\item Its isomorphism class, together with the labeled
stratification, is independent of the affine realization.

\item If $u\leq v$, it is isomorphic to the ordinary Richardson
variety, compatibly with the Richardson stratification.
\end{enumerate}
\end{thmintro}

We write $\CT_{u,v}^{\mathrm{abs}}$ and
$\CT_{u,v}^{\mathrm{abs},\circ}$ for the isomorphism classes of
the closed variety and its open stratum. The construction is
available for split forms over arbitrary fields, so the same
models can be used for point counting over finite fields and
for total nonnegativity over the real numbers. Independence rests on an isomorphism theorem for Richardson
varieties in Kac--Moody flag varieties; see Section~\ref{sec:4}.

\subsection{$R$-polynomials and point counting}

Kazhdan--Lusztig $R$-polynomials are basic invariants of Hecke
algebras. They describe the bar involution in the standard basis
and determine the Kazhdan--Lusztig polynomials recursively
\cite{KL79}. The latter encode intersection cohomology of
Schubert varieties and play a central role in representation
theory. The $R$-polynomials themselves have a direct geometric
interpretation: their values over finite fields count the
rational points of open Richardson varieties. They therefore
connect Hecke algebra structure with the geometry of Bruhat
intervals.

Our construction gives tilted $R$-polynomials with the same
point-counting interpretation. We also obtain a tilted analogue
of Dyer's increasing-path formula \cite{Dyer93}, which expresses
ordinary $R$-polynomials in terms of paths in the Bruhat graph.

\begin{thmintro}
\label{thm:intro-R-poly}
For each $u,v\in W_0$, there is a polynomial
$R_{u,v}^{\mathrm{tilt}}$ such that, over every finite field,
\[
 \#\CT_{u,v}^{\mathrm{abs},\circ}(\BF_q)
 =R_{u,v}^{\mathrm{tilt}}(q).
\]
This polynomial admits a formula in terms of label-increasing
minimal weight paths in the tilted Bruhat graph. When $u\leq v$,
it is the ordinary $R$-polynomial, and the formula reduces to
Dyer's formula.
\end{thmintro}

The precise path formula is given in Theorem~\ref{thm:R-poly}.
It computes the point counts of our affine models using a finite
graph determined by the root system. The proof transports
Dyer's formula for affine $R$-polynomials through the affine
realization. The key step is to identify the relevant paths and
construct compatible reflection orders, so that increasing paths
on the two sides correspond.

In type $A$, we conjecture that our tilted $R$-polynomials
coincide with those of Gao--Gao--Gao; see
Conjecture~\ref{conj:R-polynomial-comparison}.

\subsection{Total positivity and geometric realization}
\label{sec:intro-total-positivity}

Lusztig described totally nonnegative flag varieties as
``remarkable polyhedral spaces.'' Understanding this geometry
requires control of both the cells and their boundaries.
Rietsch proved that the intersections with open Richardson
varieties are cells and determined their closure relations
\cite{Riet99,Riet06}. Galashin--Karp--Lam established the
regularity theorem: these cell closures are closed balls and
form a regular CW complex \cite{GKL22}. This gives natural
geometric realizations of the posets of Bruhat subintervals,
in the spirit of Bj\"orner's geometric realization question
\cite{Bjor84}.

Passing to affine flag varieties brings additional difficulties.
The ambient space is infinite-dimensional, and the symmetry
between opposite Borel subgroups used in the finite case is
no longer available. The joint work of Bao and the first-named author \cite{BH24} developed product structures for
totally positive Richardson varieties and established regularity
for Kac--Moody flag varieties. This provides
the topological input for our affine models.

Gao--Gao--Gao proposed a counterpart of Bj\"orner's question for
tilted Bruhat intervals through their totally nonnegative tilted
Richardson varieties in type $A$ \cite{GGG26}. They established
a CW structure; the expected closure relations and regularity
remain conjectural. For our affine models, we prove that the
isomorphisms comparing different realizations preserve every
labeled positive cell. Combining this compatibility with the
regularity theorem of the totally nonnegative affine flag varieties gives the following result.

\begin{thmintro}
\label{thm:intro-TNN}
The totally nonnegative part of an abstract tilted Richardson
variety is a regular CW complex homeomorphic to a closed ball.
Its nonempty cells are indexed by the subintervals of the
corresponding tilted Bruhat interval, with closure relations
given by inclusion. 
\end{thmintro}

Thus the geometry of Kac--Moody flag varieties provides uniform
realizations of the posets of tilted Bruhat subintervals in all
Lie types. It supplies geometric models for a problem formulated
entirely in terms of finite Weyl groups.

\subsection{Quantum Schubert calculus}
The cohomology class of an ordinary Richardson variety is a
product of Schubert classes. Our next result extends this
relation to the minimal-degree coefficients of quantum
Schubert products. These minimal degrees are determined by
the quantum Bruhat graph, by a theorem of Postnikov
\cite{Pos05}.

To compare our affine models with classes on the finite flag
variety, we work over $\BC$ and use a continuous map to $G/B$
obtained from loop evaluation and a Weyl group action.
Let $\sigma^w$ denote the class of the opposite Schubert variety,
let $w_0$ be the longest element of $W_0$, and write $\star$
for the quantum product.

\begin{thmintro}
\label{thm:intro-quantum-class}
Under this evaluation map, the fundamental class of an affine
model for $\CT_{u,v}^{\mathrm{abs}}$ pushes forward to the
Poincar\'e dual of the minimal-degree coefficient of
\[
 \sigma^u\star\sigma^{w_0v}.
\]
In particular, the pushforward class is independent of the
choices in the affine realization.
\end{thmintro}

The precise equality in integral homology is given in
Theorem~\ref{thm:quantum-class}. Its proof compares the affine
and quantum Chevalley formulas along corresponding Bruhat
chains, using Postnikov's minimal-degree theorem to extract
the relevant quantum coefficient.

\subsection{Realization in the finite flag variety} 

The evaluation map also produces compact subsets of the finite
flag variety. In type $A$, our fundamental-class formula agrees
with the cohomology-class formula for the tilted Richardson
varieties of Gao--Gao--Gao \cite{GGG26}. This agreement motivates
the following geometric comparison.

\begin{conjintro}
\label{conj:intro-realization-in-flag}
For fixed $u,v$, the evaluation images of the affine models
coincide and form a closed algebraic subvariety of $G/B$.
In type $A$, this subvariety is the tilted Richardson variety
of Gao--Gao--Gao.
\end{conjintro}

\medskip\noindent\textbf{Organization of the paper.}
Section~\ref{sec:preliminary} recalls the necessary background.
Section~\ref{sec:3} realizes tilted Bruhat intervals as affine
Bruhat intervals, and Section~\ref{sec:4} constructs the abstract
tilted Richardson varieties and proves independence of the
realization. Sections~\ref{sec:R-polynomials} and \ref{sec:TNN}
establish the path formula and the regularity theorem.
Section~\ref{sec:quantum-cohomology} proves the fundamental-class
formula and discusses the conjectural realization in the finite
flag variety.

\medskip

{\bf Acknowledgements: } XH is partially supported by the New Cornerstone Science Foundation through the New Cornerstone Investigator Program and the Xplorer Prize, and by Hong Kong RGC grant 14300123. QY is partially supported by the National Natural Science Foundation of China (grant no. 12501018). We thank Thomas Lam and Lauren Williams for suggesting the problem. This manuscript was prepared with the assistance of AI tools to improve linguistic readability.

\section{Preliminaries}\label{sec:preliminary}

\subsection{Reductive groups}\label{sec:group-conventions}

Fix a split, simple, and simply connected Chevalley group scheme $\mathbf G_{\mathbb Z}$ together with a pinning $(\mathbf T_{\mathbb Z},\mathbf B^+_{\mathbb Z},\mathbf B^-_{\mathbb Z},(x_{i,\mathbb Z},y_{i,\mathbb Z})_{i\in I})$, where $I$ is the index set of the simple roots. Let $k$ be an arbitrary field. Set
$G_k=\mathbf G_{\mathbb Z}\times_{\Spec\mathbb Z}\Spec k$. Define $T_k$, $B_k^+$, $B_k^-$, $x_{i,k}$ and $y_{i,k}$ similarly by base change. 
Once the ground field is understood, we write $G=G_k$, $T=T_k$, $B=B_k^+$, $B^-=B_k^-$, $x_i=x_{i,k}$ and $y_i=y_{i,k}$. Unless explicitly stated otherwise, all group schemes, ind-schemes, varieties and morphisms are understood to be defined over $k$. For a field extension $K/k$ and a $k$-scheme $X$, we write $X_K=X\times_kK$. We use $G(k)$ for the group of $k$-rational points.

Let $U^+$ and $U^-$ denote the unipotent radicals of $B$ and $B^-$, respectively. Let $X^*(T)=\operatorname{Hom}(T,\mathbb G_m)$ denote the character lattice of $T$. Let $X_*(T)=\operatorname{Hom}(\mathbb G_m,T)$ denote the cocharacter lattice of $T$. These lattices are equipped with the natural perfect pairing $\langle\cdot,\cdot\rangle\colon X_*(T)\times X^*(T)\to\mathbb Z$.

Let $\Phi\subset X^*(T)$ be the root system of $G$ with respect to $T$, and let $\Phi^\vee\subset X_*(T)$ be the dual root system of coroots. The choice of the standard Borel subgroup $B$ determines a canonical set of positive roots $\Phi^+\subset\Phi$; set $\Phi^-=-\Phi^+$. The set of simple roots is $\D_0=\{\a_i\mid i\in I\}$. For each $\a\in\Phi$, denote by $\a^\vee$ the corresponding coroot. Let $\rho=\frac{1}{2}\sum_{\alpha\in\Phi^+}\alpha$ be the half-sum of the positive roots.

Let $W_0=N_G(T)/T$ be the finite Weyl group of $G$. For each $i\in I$, denote by $s_i=s_{\a_i}$ the corresponding simple reflection in $W_0$. The set of simple reflections $\BS_0=\{s_i\mid i\in I\}$ generates $W_0$, equipping $(W_0,\BS_0)$ with the structure of a finite Coxeter system. We denote the standard Coxeter length function by $\ell\colon W_0\to\mathbb Z_{\geq 0}$ and the Bruhat order on $W_0$ by $\leq$. The unique longest element of $W_0$ is denoted by $w_0$.

We use the dominance order on $Q^\vee=\BZ\Phi^\vee$: for $\lambda,\lambda'\in Q^\vee$, we write $\lambda\leq\lambda'$ if $\lambda'-\lambda$ is a nonnegative integral combination of positive coroots.

For $i\in I$, set $\dot s_i=x_i(-1)y_i(1)x_i(-1)\in G(k)$. For any $w\in W_0$, choose a \emph{reduced expression} $w=s_{i_1}s_{i_2}\cdots s_{i_n}$. We define the lift $\dot w\in G(k)$ by $\dot w=\dot s_{i_1}\dot s_{i_2}\cdots\dot s_{i_n}$. This definition is independent of the choice of reduced expression, and it provides the canonical representative lift $\dot w\in N_G(T)(k)$ used in \S\ref{sec:flag-richardson}.

The algebraic construction is carried out for split simple, simply connected groups over an arbitrary field. One may deduce the results for arbitrary split connected reductive groups by a standard procedure. For total nonnegativity we use the induced pinnings on split real forms.

\subsection{Flag varieties and Richardson varieties}\label{sec:flag-richardson}
Let $\CB = G/B$ be the flag variety. All intersections in this subsection are scheme-theoretic. The flag variety decomposes into a disjoint union of $B$-orbits, called \emph{Schubert cells}:
\[ \mathcal{B} = \bigsqcup_{w \in W_0} \mathring{\mathcal{B}}_w, \quad \text{where} \quad \mathring{\mathcal{B}}_w := B \dot w B/B. \] 
Each Schubert cell $\mathring{\mathcal{B}}_w$ is isomorphic to the affine space $\mathbb A_k^{\ell(w)}$. The Zariski closure of a Schubert cell is the \emph{Schubert variety}, denoted by $\mathcal{B}_w = \overline{\mathring{\mathcal{B}}_w}$. The classical Bruhat order geometrically encodes the topological closure relations among these varieties: $\mathcal{B}_v \subseteq \mathcal{B}_w$ if and only if $v \le w$ in $W_0$. Thus, $\mathcal{B}_w = \bigsqcup_{v \le w} \mathring{\mathcal{B}}_v$.

We consider the orbits of the opposite Borel subgroup $B^-$ acting on $\mathcal{B}$. For any $u \in W_0$, the \emph{opposite Schubert cell} is defined as:
\[ \mathring{\mathcal{B}}^u := B^- \dot u B/B. \]
This cell is isomorphic to an affine space of codimension $\ell(u)$ in $\mathcal{B}$. Its Zariski closure is the \emph{opposite Schubert variety} $\mathcal{B}^u = \overline{\mathring{\mathcal{B}}^u}$. The closure relations for opposite Schubert varieties are the reverse of those of Schubert varieties: $\mathcal{B}^w \subseteq \mathcal{B}^u$ if and only if $u \le w$. 

By intersecting these dual stratifications, one obtains Richardson varieties. For any pair of elements $u, v \in W_0$, the \emph{open Richardson variety} is defined as the transverse intersection of an opposite Schubert cell and a Schubert cell:
\[ \mathring{\mathcal{B}}_{u,v} := \mathring{\mathcal{B}}^u \cap \mathring{\mathcal{B}}_v. \]
Note that $\mathring{\mathcal{B}}_{u,v}$ is non-empty if and only if $u \leq v$ (cf. \cite{KL79}). In this case, since the $B$-orbits and $B^-$-orbits intersect transversally, the open Richardson variety $\mathring{\mathcal{B}}_{u,v}$ is a smooth, locally closed, irreducible subvariety of $\mathcal{B}$ (cf. \cite{Ric92}). Its Zariski closure in $\mathcal{B}$ is the \emph{closed Richardson variety}:
\[ \mathcal{B}_{u,v} := \mathcal{B}^u \cap \mathcal{B}_v. \]

The geometric properties of these varieties reflect the combinatorial structure of the Bruhat interval $[u,v]$. When $u\le v$, $\mathcal{B}_{u,v}$ is an irreducible algebraic variety of pure dimension $\ell(v) - \ell(u)$ and $$\CB_{u, v}=\bigsqcup_{u \le x \le y \le v} \mathring{\mathcal{B}}_{x, y}$$
(cf. \cite[Theorem 3.2]{Ric92}). 


\subsection{Quantum Bruhat graph and tilted Bruhat intervals}
\label{sec:2.4}
While classical Schubert calculus is governed by the Bruhat order, extending this intersection theory to the quantum cohomology ring requires a richer combinatorial structure that accounts for rational curves of strictly positive degree. The quantum Bruhat graph, introduced by Fomin, Gelfand and Postnikov in \cite{FGP97}, encodes the quantum Chevalley formula. Set $\rho = \frac{1}{2}\sum_{\a\in \Phi^+} \a$. Let $w_0$ be the longest element in $W_0$. The quantum Bruhat graph $\QBG(\Phi)$ is a directed graph with:
\begin{enumerate}
    \item Vertices: Elements of $W_0$.
    \item Edges: 
\begin{enumerate}[(i)]
    \item Bruhat edges: $w \rightharpoonup w s_\alpha$ when $\ell(w s_\alpha) = \ell(w) + 1$, $\a \in \Phi^+$.
    \item Quantum edges: $w \rightharpoondown w s_\alpha$ when $\ell(w s_\alpha) = \ell(w) + 1 - \langle \alpha^\vee, 2\rho \rangle$, $\a\in\Phi^+$, and $\a^{\vee}$ is the dual coroot to $\a$.
\end{enumerate}
     \item Weights: A Bruhat edge $w \rightharpoonup w s_\alpha$ has weight 0; a quantum edge $w \rightharpoondown w s_\alpha$ has weight $\alpha^\vee$.
\end{enumerate}

\begin{example}
In the case of type $A_2$, the quantum Bruhat graph has six vertices.
See Figure~\ref{fig:1}.

\begin{figure}[H]
\centering
\def\seshift{0.5ex}
\begin{tikzcd}[ampersand replacement=\&,column sep=2em]
  \&
  s_1s_2s_1
    \ar[ddd]
    \ar[dl,shift right=\seshift]
    \ar[dr,shift left=\seshift]
  \\
  s_1s_2
    \ar[ur,shift right=\seshift]
    \ar[d,shift right=\seshift]
  \& \&
  s_2s_1
    \ar[ul,shift left=\seshift]
    \ar[d,shift left=\seshift]
  \\
  s_1
    \ar[u,shift right=\seshift]
    \ar[urr]
    \ar[dr,shift right=\seshift]
  \& \&
  s_2
    \ar[u,shift left=\seshift]
    \ar[ull]
    \ar[dl,shift left=\seshift]
  \\
  \&
  1
    \ar[ru,shift left=\seshift]
    \ar[lu,shift right=\seshift]
\end{tikzcd}
\caption{Quantum Bruhat Graph for $A_2$}
\label{fig:1}
\end{figure}
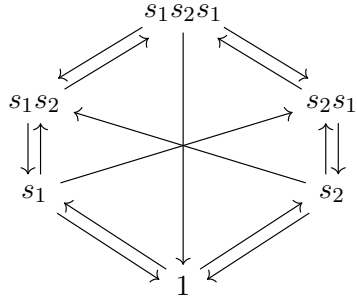
\end{example}

For any $u \in W_0$ and any simple reflection $s_\a \in \BS_0$, it is clear that $u \to us_\a$ is a (Bruhat or quantum) edge of $\QBG(\Phi)$. Then, one can easily deduce that $\QBG(\Phi)$ is path-connected. For $u,v\in W_0$, the length of any shortest path from $u$ to $v$ is denoted by $\ell_{\QBG}(u,v)$, or simply $\ell(u,v)$. The {\it weight} of a path $p$ is the sum of the weights of its edges, denoted by $\wt(p)$. 

The following results are due to Postnikov \cite[Lemma 1]{Pos05} and Mili\'cevi\'c--Viehmann \cite[(4.3)]{MV20}. 
\begin{lemma}\label{lem:QBG}
Let $u,v\in W_0$. Then 
\begin{enumerate}
    \item All shortest paths from $u$ to $v$ in $\mathrm{QBG}(\Phi)$ have the same weight. We denote this weight by $\wt_{\QBG}(u, v)$, or simply $\wt(u, v)$.
    \item Any path from $u$ to $v$ has weight $\ge \wt(u,v)$; the equality holds if and only if the path is shortest.
\end{enumerate}
\end{lemma}

The weight function satisfies the following properties.
\begin{lemma}\label{lem:QBG2}
Let $u,u',u''\in W_0$. Then,
\begin{enumerate}
\item  $\ell(u,u') - \<\wt(u,u'), 2\rho\> = \ell(u') - \ell(u)$.

\item $\wt(u,u')=0$ if and only if $u\le u'$.

\item  $\wt(u,u'') \le \wt(u,u')+ \wt(u',u'')$.

\item $\<\wt(u,u'),\rho\> \le \ell(w_0)$.
\end{enumerate}    
\end{lemma}
\begin{proof}
(1) follows from induction on $\ell(u,u')$. (2) and (3) follow from Lemma \ref{lem:QBG}. Using (1), we deduce that $\<\wt(u,u'),2\rho\> = \ell(u,u') + \ell(u) - \ell(u') \le 2\ell(w_0)$. This proves (4).
\end{proof}

We recall the definition of the tilted Bruhat order, introduced in {\cite[\S6]{BFP}}.  

\begin{definition}\label{def:tilted-Bruhat-order}
Let $x \in W_0$. 
\begin{enumerate}
\item The \emph{tilted Bruhat order} $\le_{x}$ is defined as follows. For $u,v \in W_0$, define $u \le_{x} v$ if $\wt(x,u) + \wt(u,v) = \wt(x,v)$. 

\item The \emph{opposite tilted Bruhat order} $\le_{x}^{\mathrm{opp}}$ is defined as follows. For $u,v \in W_0$, define $u \le_{x}^{\mathrm{opp}} v$ if $\wt(u,v) + \wt(v,x) = \wt(u,x)$. 

\end{enumerate}
\end{definition}

By Lemma \ref{lem:QBG} (2), $u \le_{x} v$ if and only if there is a shortest path from $x$ to $v$ passing through $u$. In the case of $x = 1$, the order $\le_{x}$ coincides with the usual Bruhat order $\le$. Similarly, $u \le_{x}^{\mathrm{opp}} v$ if and only if there is a shortest path from $u$ to $x$ passing through $v$. When $x$ is the longest element in $W_0$, the order $\le_{x}^{\mathrm{opp}}$ coincides with the usual Bruhat order $\le$.

\begin{remark}
The poset $(W_0,\le_x)$ has a unique minimal element $x$. But it may not have a unique maximal element. For example, in the case of type $A_2$ and $x = s_1s_2s_1$, $(W_0,\le_x)$ has two maximal elements $s_1$ and $s_2$ (Figure \ref{fig:1}).
\end{remark}

\begin{definition}\label{def:tilted-bruhat-interval}
Let $u,v \in W_0$. The \emph{tilted Bruhat interval} $[u,v]_{\mathrm{tilt}}$, or $[u,v]_{\QBG}$, is defined as
\begin{align*}
[u,v]_{\mathrm{tilt}} = \{ z \in W_0 \mid \wt(u,z) + \wt(z, v) = \wt(u,v) \}.
\end{align*}
Define a partial order $\le_{u,v}$ as follows. For $u',v' \in [u,v]_{\mathrm{tilt}}$, define $u' \le_{u,v} v'$ if $\wt(u,u') + \wt(u',v') + \wt(v',v) = \wt(u,v)$. In this case, we say $[u',v']_{\tilt}$ is a subinterval of $[u,v]_{\tilt}$, denoted by $[u',v']_{\tilt} \subset [u,v]_{\tilt}$.
\end{definition}

By Lemma \ref{lem:QBG} (2), $z \in [u,v]_{\mathrm{tilt}}$ if and only if there is a shortest path from $u$ to $v$ passing through $z$. In the case of $u \le v$, by Lemma \ref{lem:QBG} (2), the tilted Bruhat interval $[u,v]_{\mathrm{tilt}}$ is equal to the ordinary Bruhat interval $[u,v]$. Note also that for $u',v'\in [u,v]_{\mathrm{tilt}}$, $u' <_{u,v} v'$ if and only if there is a shortest path from $u$ to $v$ passing through $u'$ first and $v'$ later.


The following lemma follows directly from the definition.

\begin{lemma}\label{lem:tilted-interval}
Let $u,v \in W_0$. 
\begin{enumerate}
\item For any $x\in W_0$ such that $u \le_x v$, $$([u,v]_{\tilt},\le_{u,v}) = (\{z \in W_0 \mid u\le_x z \le_x v\},\le_x).$$

\item For any $x\in W_0$ such that $u \le_x^{\mathrm{opp}} v$, $$([u,v]_{\tilt},\le_{u,v}) = (\{z \in W_0 \mid u\le_x^{\mathrm{opp}} z \le_x^{\mathrm{opp}} v\},\le_x^{\mathrm{opp}}).$$
\end{enumerate}
\end{lemma}
Note that the element $x$ in Lemma \ref{lem:tilted-interval} always exists. Indeed, we have $u\le_u v$ and $u\le_v^{\mathrm{opp}}v$. Lemma \ref{lem:tilted-interval} indicates that the intervals $\{z \mid u\le_x z \le_x v\}$ and $\{z \mid u\le_x^{\mathrm{opp}} z \le_x^{\mathrm{opp}} v\}$ are independent of the choices of $x$.

\section{Affine realizations of tilted Bruhat intervals}\label{sec:3}

\subsection{The affine Weyl group}\label{sec:3.1}

Let $Q^{\vee} = \BZ \Phi^{\vee}$ be the coroot lattice. A cocharacter $\l \in Q^{\vee}$ is said to be \emph{dominant} if $\<\l,\a\> \ge 0$ for all $\a \in \Phi^+$. A cocharacter $\l \in Q^{\vee}$ is said to be \emph{dominant regular} if $\<\l,\a\> > 0$ for all $\a \in \Phi^+$. 

By convention, the set of affine roots is defined as $\Phi_{\aff} = \{\a + m\d \mid \a \in \Phi, m \in \BZ\}$, where $\d$ is the null root. Let $\tilde I=I \sqcup \{0\}$. We set $\ta_i = \a_i$ for $i\in I$ and $\ta_0= -\th + \d$. The set of affine simple roots is $\tilde \D=\{\ta_i; i \in \tilde I\}$. 

Let $W_{\aff}=Q^\vee \rtimes W_0$ denote the affine Weyl group corresponding to $G$. It is a Coxeter group with the set of simple reflections $\tilde \BS=\BS_0 \sqcup \{s_0\}$, where $s_0=t^{-\th^{\vee}}s_{\th}$. Here, $\th$ is the highest root of $\Phi$, $\th^{\vee}$ is the coroot dual to $\th$, and $s_\th$ is the reflection in $W_0$ corresponding to $\th$. 

Let $\le$ and $\ell$ be the Bruhat order and the length function of the Coxeter group $W_{\aff}$, respectively. For $\l\in Q^{\vee}$ and $z \in W_0$, based on our choice of positive affine roots, we have the exact length formula:
\begin{align}\label{eq:len}
    \ell(t^{\l} z) = \sum_{\a \in \Phi^+, z^{-1}(\a) \in \Phi^+} |\< \l,\a\> | + \sum_{\a \in \Phi^+,z^{-1}(\a) \in \Phi^-} |\< \l,\a\> + 1|. 
\end{align}
For $w_1,w_2\in W_{\aff}$, denote $[w_1,w_2]=\{w \in W_{\aff} \mid w_1 \le w \le w_2 \}$.

Note that any element $w\in W_{\aff}$ can be written as $w = x t^{-\l} y$ with $x, y \in W_0$ and $\l$ a dominant cocharacter. However, the presentation $w = x t^{-\l} y$ is not unique in general. Indeed, for any $u \in W_0$ such that $u(\l)=\l$, we have $w = x u t^{-\l} u^{-1} y$. If $\l$ is dominant regular, then the presentation $w = x t^{-\l} y$ is unique. In this case, by evaluating the length formula \eqref{eq:len}, we obtain:
\begin{align}\label{eq:len2}
    \ell(w) = \ell(x) + \<\l,2\rho\> - \ell(y).
\end{align}

Lam--Shimozono \cite[Proposition 4.4]{LS10} proved a correspondence between covering relations in $W_{\aff}$ and edges in $\QBG(\Phi)$ for sufficiently regular elements. Schremmer proves the following deeper result {\cite[Theorem 4.2]{Sch24}} (with a different convention).

\begin{theorem}\label{thm:QBG}
Let $\l,\l'\in Q^{\vee}$ be such that $\<\l,\a\> \ge 2$ and $\<\l',\a\> \ge 2$ for any $\a\in \Phi^+$. Let $w = x t^{-\l}y$ and $w' = x' t^{-\l'} y'$ with $x,y,x',y'\in W_0$. Then $w \le w'$ if and only if
\begin{align*}
\wt(x, x') + \wt(y'^{-1}, y^{-1}) \le \l' - \l.
\end{align*}
\end{theorem}

\subsection{Affine realizations of tilted Bruhat intervals}
A cocharacter $\mu \in Q^{\vee}$ is said to be \emph{sufficiently dominant regular} if $\<\mu,\a\> \ge 3\ell(w_0)+3$ for any $\a \in \Phi^+$. Here, $w_0$ is the longest element in $W_0$. By Lemma \ref{lem:QBG2} (4), for a sufficiently dominant regular cocharacter $\mu$, we have
\begin{align}\label{eq:sufficiently}
    \<\mu \pm \wt(u,v),\a\> \ge 2,\quad\text{for any }u,v\in W_0,\a\in\Phi^+.
\end{align}
In particular, $\mu \pm \wt(u,v)$ satisfies the condition in Theorem \ref{thm:QBG}. The bound $3\ell(w_0)+3$ also ensures the regularity of the intermediate elements used in the proof of Proposition \ref{prop:realization}.

Let $x\in W_0$ and let $\mu\in Q^{\vee}$ be sufficiently dominant regular. We use the following \emph{type $\mathrm{I}$ affine realization map}:
\begin{equation}\label{eq:psi}
\begin{aligned}
\Psi_{\mu,x}=\Psi_{\mu,x}^{\mathrm{I}}:(W_0,\le_x)&\longrightarrow (W_{\aff},\le),\\
z&\longmapsto w_{\mu,x}(z)=w_{\mu,x}^{\mathrm{I}}(z)
:=zt^{-(\mu+\wt(x,z))}x^{-1}.
\end{aligned}
\end{equation}
The other affine realizations are recorded at the end of this section.

\begin{proposition}
\label{prop:realization}
Let $x\in W_0$ and let $\mu\in Q^{\vee}$ be sufficiently dominant regular. Then $\Psi_{\mu,x}$ is an injective morphism of posets. For any $u,v\in W_0$ with $u\le_xv$, it induces a poset isomorphism
\[
([u,v]_{\mathrm{tilt}},\le_{u,v})
\xrightarrow{\sim}
([w_{\mu,x}(u),w_{\mu,x}(v)],\le).
\]
\end{proposition}

\begin{proof}
It is clear that $\Psi_{\mu,x}$ is injective.

Since $\mu$ is sufficiently dominant regular, by (\ref{eq:sufficiently}), for any $z\in W_0$, $\mu +\wt(x,z)$ satisfies the regularity condition in Theorem \ref{thm:QBG}. Then, for $z_1,z_2 \in W_0$, Theorem \ref{thm:QBG} and Lemma \ref{lem:QBG2} (3) imply that $w_{\mu,x}(z_1)\le w_{\mu,x}(z_2)$ if and only if $z_1 \le_x z_2$.


Let $u,v\in W_0$ be such that $u\le_x v$. By Lemma \ref{lem:tilted-interval}, we have $([u,v]_{\mathrm{tilt}},\le_{u,v}) = (\{z\mid u\le_x z \le_x v\}, \le_x)$. Then $\Psi_{\mu,x}$ induces an injective poset morphism $([u,v]_{\mathrm{tilt}},\le_{u,v})\to ([w_{\mu,x}(u),w_{\mu,x}(v)],\le)$. It remains to prove that the image is $[w_{\mu,x}(u),w_{\mu,x}(v)]$. Let $w \in W_{\aff}$ be such that $w_{\mu,x}(u) \le w \le w_{\mu,x}(v)$. Write $w = a t^{-\l} b$ with $\l$ dominant. Using the double coset argument, it is clear that the anti-dominant translation part $-\l$ of $w$ lies between those of $w_{\mu,x}(u)$ and $w_{\mu,x}(v)$. By (\ref{eq:sufficiently}), we see that Theorem \ref{thm:QBG} applies to $w$. 

Therefore, we get
\begin{align*}
    \wt(u,a) + \wt(b^{-1},x) \le \l-\mu - \wt(x,u)\\
    \wt(a,v) + \wt(x,b^{-1}) \le \mu + \wt(x,v) - \l.
\end{align*}
Summing up these two inequalities and using the condition $u\le_x v$, we get
$$\wt(u,a) + \wt(a,v) + \wt(b^{-1},x) + \wt(x,b^{-1})\le \wt(u,v).$$
This forces all equalities to hold. In particular, we have $\wt(u,a) + \wt(a,v) = \wt(u,v)$, $b = x^{-1}$ and $\l = \mu + \wt(x,a)$. That is, $a \in [u,v]_{\tilt}$ and $w = w_{\mu,x}(a)$. This completes the proof.
\end{proof}

\begin{example}
In type $A_2$, take $x=u=s_1s_2$ and $v=s_2s_1$.
The affine realization gives an isomorphism of posets
\[
([s_1s_2,s_2s_1]_{\mathrm{tilt}},\le_{u,v})
\cong
\bigl([s_1s_2t^{-\mu}s_2s_1,
s_2s_1t^{-(\mu+\alpha_2^\vee)}s_2s_1],\le\bigr).
\]
Figure~\ref{fig:A2-tilted-affine} displays the two intervals with
corresponding vertices in the same positions. On the left, solid arrows
are Bruhat edges and dashed arrows are quantum edges. On the right, all arrows are affine Bruhat covers:
the solid arrows preserve the translation part in the displayed
expressions, whereas the dashed arrows change the translation part.

\begin{figure}[htbp]
\centering
\makebox[\textwidth][l]{%
\hspace*{9mm}
\begin{minipage}{\textwidth}
\centering

\begin{minipage}[t]{0.30\textwidth}
\vspace{0pt}
\centering
\begin{tikzpicture}[
  >=Stealth,
  every node/.style={
    inner sep=3pt,
    minimum height=7mm,
    font=\small
  },
  bruhat/.style={->,semithick},
  quantum/.style={bruhat,dashed}
]
  \node (u) at (0,0)
    {$s_1s_2$};
  \node (a) at (-1.15,1.35)
    {$s_1$};
  \node (b) at (1.15,1.35)
    {$s_1s_2s_1$};
  \node (v) at (0,2.7)
    {$s_2s_1$};

  \draw[quantum] (u) -- (a);
  \draw[bruhat]  (a) -- (v);
  \draw[bruhat]  (u) -- (b);
  \draw[quantum] (b) -- (v);
\end{tikzpicture}

\end{minipage}\hfill
\begin{minipage}[t]{0.68\textwidth}
\vspace{0pt}
\centering
\begin{tikzpicture}[
  >=Stealth,
  every node/.style={
    inner xsep=2pt,
    inner ysep=3pt,
    minimum height=7mm,
    font=\small
  },
  bruhat/.style={->,semithick},
  translated/.style={bruhat,dashed}
]
  \node (u) at (0,0)
    {$s_1s_2t^{-\mu}s_2s_1$};
  \node (a) at (-1.55,1.35)
    {$s_1t^{-(\mu+\alpha_2^\vee)}s_2s_1$};
  \node (b) at (1.55,1.35)
    {$s_1s_2s_1t^{-\mu}s_2s_1$};
  \node (v) at (0,2.7)
    {$s_2s_1t^{-(\mu+\alpha_2^\vee)}s_2s_1$};

  \draw[translated]
    ([xshift=-3.5mm]u.north) --
    ([xshift= 3.5mm]a.south);

  \draw[bruhat]
    ([xshift= 3.5mm]a.north) --
    ([xshift=-3.5mm]v.south);

  \draw[bruhat]
    ([xshift= 3.5mm]u.north) --
    ([xshift=-3.5mm]b.south);

  \draw[translated]
    ([xshift=-3.5mm]b.north) --
    ([xshift= 3.5mm]v.south);
\end{tikzpicture}

\end{minipage}

\end{minipage}%
}
\caption{An example of affine realization}
\label{fig:A2-tilted-affine}
\end{figure}
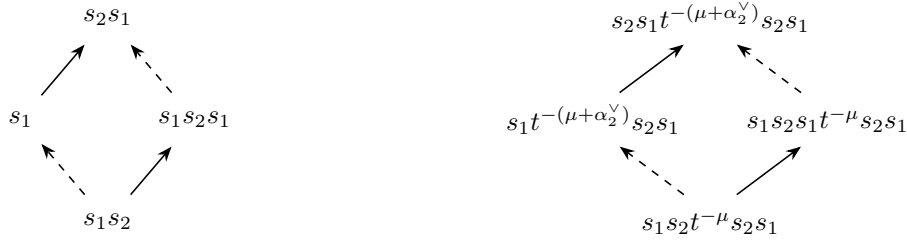
\end{example}

Besides the map $\Psi_{\mu,x}^{\mathrm{I}}$ in \eqref{eq:psi}, we consider the following three variants:
\begin{equation}\label{eq:other-realizations}
\begin{aligned}
\Psi_{\mu,x}^{\mathrm{II}}:(W_0,\le_x) \longrightarrow (W_{\aff},\le^{\opp}), \quad
& z \longmapsto w_{\mu,x}^{\mathrm{II}}(z):=xt^{-(\mu-\wt(x,z))}z^{-1}, \\
\Psi_{\mu,x}^{\mathrm{III}}:(W_0,\le_x^{\opp})\longrightarrow (W_{\aff},\le),\quad
& z \longmapsto w_{\mu,x}^{\mathrm{III}}(z):=zt^{-(\mu-\wt(z,x))}x^{-1},\\
\Psi_{\mu,x}^{\mathrm{IV}}:(W_0,\le_x^{\opp}) \longrightarrow (W_{\aff},\le^{\opp}),\quad 
& z \longmapsto w_{\mu,x}^{\mathrm{IV}}(z):=xt^{-(\mu+\wt(z,x))}z^{-1}.
\end{aligned}
\end{equation}
Here $\le_x^{\opp}$ is the opposite tilted Bruhat order of Definition~\ref{def:tilted-Bruhat-order}(2), and $\le^{\opp}$ is the opposite Bruhat order on $W_{\aff}$.

The same argument as in the proof of Proposition~\ref{prop:realization} gives the following.
\begin{proposition}
Let $x\in W_0$ and let $\mu\in Q^{\vee}$ be sufficiently dominant regular. The three maps in \eqref{eq:other-realizations} are injective morphisms of posets. Moreover:
\begin{enumerate}
\item If $u\le_xv$, then $\Psi_{\mu,x}^{\mathrm{II}}$ induces a poset isomorphism
\[
([u,v]_{\mathrm{tilt}},\le_{u,v})
\xrightarrow{\sim}
([w_{\mu,x}^{\mathrm{II}}(v),w_{\mu,x}^{\mathrm{II}}(u)],\le^{\opp}).
\]
\item If $u\le_x^{\opp}v$, then $\Psi_{\mu,x}^{\mathrm{III}}$ and $\Psi_{\mu,x}^{\mathrm{IV}}$ induce poset isomorphisms
\[
\begin{aligned}
([u,v]_{\mathrm{tilt}},\le_{u,v})
&\xrightarrow{\sim}
([w_{\mu,x}^{\mathrm{III}}(u),w_{\mu,x}^{\mathrm{III}}(v)],\le),\\
([u,v]_{\mathrm{tilt}},\le_{u,v})
&\xrightarrow{\sim}
([w_{\mu,x}^{\mathrm{IV}}(v),w_{\mu,x}^{\mathrm{IV}}(u)],\le^{\opp}).
\end{aligned}
\]
\end{enumerate}
\end{proposition}

\begin{remark}
The isomorphism $\Psi_{\mu,x}^{\mathrm{II}}$: $([u,v]_{\mathrm{tilt}},\le_{u,v}) \cong ([w_{\mu,x}^{\mathrm{II}}(v),w_{\mu,x}^{\mathrm{II}}(u)],\le^{\opp})$ is essentially proved in \cite[Theorem 4.8]{LS10} (with a complicated change of convention)\footnote{ Indeed, a presentation $w  t^{\lambda+\wt(u,w)}u^{-1}$, $\l$ sufficiently anti-dominant in \cite[Theorem 4.8]{LS10} corresponds to the presentation $u t^{-(-\l - \wt(u,w))} w^{-1}$ in our convention. The major difference is that in \cite{LS10}, $W_{\aff}$ acts on the set of chambers on the right, while in this paper, $W_{\aff}$ acts on the left.}. Types $\mathrm{I}$ and $\mathrm{III}$ correspond to ``far covers'' in \cite[Section 4]{LS10}, while types $\mathrm{II}$ and $\mathrm{IV}$ correspond to ``near covers''. We give a different proof here using Theorem \ref{thm:QBG}.
\end{remark}

\section{Abstract tilted Richardson varieties}
\label{sec:4}
\subsection{Richardson varieties in the affine flag variety}\label{sec:4.1}

Define $LG(R)=G(R((\e)))$,
$L^+G(R)=G(R[[\e]])$ and $L^-G(R)=G(R[\e^{-1}])$, for any $k$-algebra $R$. Let $\operatorname{ev}_0:L^+G\to G$ be the evaluation map at $\e=0$. Let $\operatorname{ev}_\infty:L^-G\to G$ be the evaluation map at
$\e^{-1}=0$. The standard Iwahori subgroup is
$\CI=\operatorname{ev}_0^{-1}(B)$. The opposite Iwahori subgroup is
$\CI^-=\operatorname{ev}_\infty^{-1}(B^-)$.

Recall that the affine simple roots are $\ta_i = \a_i$, $i\in I$, and $\ta_0= -\th + \d$. Recall that $\tilde I = I \sqcup \{0\}$ (see \S\ref{sec:3.1}). For each $i\in\widetilde I$, fix the standard rank-one homomorphism
\begin{align*}
\varphi_i:\mathrm{SL}_{2,k}\longrightarrow LG
\end{align*}
associated with the affine simple root $\widetilde\alpha_i$. Define
\begin{align*}
\widetilde x_i(a)
=
\varphi_i\!\begin{pmatrix}1&a\\0&1\end{pmatrix},
\qquad
\widetilde y_i(a)
=
\varphi_i\!\begin{pmatrix}1&0\\a&1\end{pmatrix}.
\end{align*}
Set $\dot s_i
=
\widetilde x_i(-1)\widetilde y_i(1)\widetilde x_i(-1)$. For each $w \in W_{\aff}$, we choose a reduced expression $w = s_{i_1} s_{i_2} \cdots s_{i_n}$ and define the lift $\dot{w} \in LG(k)$ by $\dot{w} = \dot{s}_{i_1} \dot{s}_{i_2} \cdots \dot{s}_{i_n}$.

Let $\widetilde{\CB} = LG /\CI$ be the affine flag variety. By \cite[9.h.]{PR08} and a faithfully flat argument, $\widetilde{\CB}$ is $k$-isomorphic as an ind-scheme to the flag variety $\widehat{G}/\widehat{B}^+$ of the associated affine Kac--Moody group, where the central extension and loop rotation parameters of the Kac--Moody group are naturally absorbed by the Iwahori subgroup quotient. This identification allows us to apply results on Kac--Moody flag varieties to the loop group setting.

As in the finite case, for $w,w'\in W_{\aff}$, the \emph{open affine Richardson variety} is defined as
\begin{align*}
\widetilde{\CB}_{w,w'}^{\circ}
=
\bigl(\CI^-\dot{w}\CI/\CI\bigr)
\cap
\bigl(\CI\dot{w}'\CI/\CI\bigr).
\end{align*}
Its Zariski closure in $\widetilde{\CB}$ is the (closed) \emph{affine Richardson variety} $\widetilde{\CB}_{w,w'}$.

The open affine Richardson variety $\widetilde{\CB}_{w,w'}^{\circ}$ is non-empty if and only if $w \le w'$. In this case, $\widetilde{\CB}_{w,w'}^{\circ}$ is an irreducible algebraic variety of pure dimension $\ell(w') - \ell(w)$. Moreover, $$\widetilde{\CB}_{w,w'}=\bigsqcup_{w_1, w_2 \in W_{\aff}; w \le w_1 \le w_2 \le w'} \widetilde{\CB}_{w_1,w_2}^{\circ}.$$

\begin{definition}\label{def:tilted-richardson}
Let $u,v\in W_0$. Choose $x\in W_0$ such that $u \le_x v$ and choose a sufficiently dominant regular cocharacter $\mu$. Define
\begin{align*}
\CT_{u,v}^{\mu,x,\circ} = \widetilde{\CB}_{w_{\mu,x}(u),w_{\mu,x}(v)}^\circ, \quad \CT_{u,v}^{\mu,x} = \widetilde{\CB}_{w_{\mu,x}(u),w_{\mu,x}(v)}.
\end{align*}
See \eqref{eq:psi} for the notation.
\end{definition}

\subsection{Isomorphism theorem for Richardson varieties}\label{sec:4.2}

A priori, the isomorphism class of the variety $\CT_{u,v}^{\mu,x}$ (and $\CT_{u,v}^{\mu,x,\circ}$) depends on the choice of $\mu$ and $x$. In the rest of this section, we will show the independence of $\mu$ and $x$.

We need the following simple lemma. A product $w_1w_2 \in W_{\aff}$ is called \emph{length-additive} if $\ell(w_1w_2) = \ell(w_1) + \ell(w_2)$.

\begin{lemma}\label{lem:length-additive}
Let $w,w',y \in W_{\aff}$ with $w \le w'$. Assume that $\ell(w''y) = \ell(w'')+\ell(y)$ for any $w'' \in [w,w']$. Then, right multiplication by $y$ induces an isomorphism of posets between $[w,w']$ and $[wy,w'y]$.
\end{lemma}

\begin{proof}
By induction on $\ell(y)$, it suffices to deal with the case where $y = s$ is a simple reflection.

The map given by right multiplication by $s$ preserves the Bruhat order: if $a \le b$ in $[w,w']$, the standard Bruhat lifting property ensures $as \le bs$. Thus, the map is a strict poset embedding of $[w,w']$ into $[ws, w's]$. To prove surjectivity, suppose there exists $a \in [ws, w's]$ that is not in the image. This means $a$ does not have $s$ as a right descent, so $\ell(as) = \ell(a)+1$. Because $a \le w's$ and $w's$ has $s$ as a right descent, the lifting property implies that $a \le w'$. Combined with the lower bound $ws \le a$, we obtain $w < ws \le a \le w'$. This implies the element $ws$ lies in the interval $[w,w']$. This contradicts the assumption that $\ell(w''s)=\ell(w'')+1$ for all $w'' \in [w,w']$.
\end{proof}

We first establish the following general result on isomorphisms of Richardson varieties. We state the result for affine Richardson varieties. The same argument applies to the flag varieties of symmetrizable Kac--Moody groups.

\begin{theorem}\label{thm:technical}
Let $w,w',y \in W_{\aff}$ with $w \le w'$. Assume that $\ell(w''y) = \ell(w'')+\ell(y)$ for any $w'' \in [w,w']$. Then there exists a natural $k$-isomorphism $\widetilde{\CB}_{w,w'} \cong \widetilde{\CB}_{wy,w'y}$ which restricts to a $k$-isomorphism $\widetilde{\CB}_{w_1,w_2}^{\circ} \cong \widetilde{\CB}_{w_1 y,w_2 y}^{\circ}$ for any $w \le w_1 \le w_2 \le w'$. 
\end{theorem}

We first recall the Deodhar parameterizations and prove a truncation lemma.

\begin{definition}[{\cite[Definition 2.3]{Deod85}, \cite[Definition~3.4]{MR04}}]
Let $w,x \in W_{\aff}$.

\begin{enumerate}
\item Let $\ud w = s_{i_1} s_{i_2} \cdots s_{i_r}$ be a reduced expression of $w$. A \emph{subexpression} for $x$ in $\ud w$ is a sequence $\ud x = (x_{(0)} = 1, x_{(1)}, x_{(2)}, \dots, x_{(r)} = x)$ such that $x_{(j)} \in \bigl\{x_{(j-1)},\, x_{(j-1)} s_{i_j}\bigr\}$ for all $j$.

\item A subexpression $\ud x$ as in (1) is called \emph{distinguished} if $x_{(j)} \leq x_{(j-1)} s_{i_j}$ for all $j$. In particular, if $x_{(j)} = x_{(j-1)}$, we must have $x_{(j-1)}s_{i_j} > x_{(j-1)}$.

\item The set of distinguished subexpressions for $x$ in $\ud{w}$ is denoted by $\mathrm{Dist}_{x,\ud{w}}$.
\item A subexpression $\ud x$ as in (1) is called \emph{positive distinguished} if $x_{(j-1)} \le x_{(j)} \le x_{(j-1)} s_{i_j}$ for all $j$.
\end{enumerate}   
\end{definition}

Let $\ud w = s_{i_1} s_{i_2} \cdots s_{i_r}$ be a reduced expression of $w \in W_{\aff}$. Let $\ud x = (x_{(0)}=1, x_{(1)}, \ldots,x_{(r)}=x)$ be a distinguished subexpression for some $x \le w$. Define
\begin{align*}
    J_{\ud x}^+ = \{j\mid x_{(j-1)} < x_{(j)}\},\\
    J_{\ud x}^{\circ} = \{j\mid x_{(j-1)} = x_{(j)}\},\\
    J_{\ud x}^- = \{j\mid x_{(j-1)} > x_{(j)}\}.
\end{align*}

Recall that we have fixed $\tilde x_{i}, \tilde y_i$, $i \in \tilde I$ in \S\ref{sec:4.1}. The following result is proved in \cite[Proposition 5.2]{MR04} for reductive groups. The same argument works for loop groups.

\begin{theorem}\label{thm:MR}
The open Richardson variety $\widetilde{\CB}^\circ_{x,w}$
admits a decomposition into Deodhar cells $\mathcal{D}_{\ud x,\ud w}$
\begin{align*}
\widetilde{\CB}^\circ_{x,w}
=
\bigsqcup_{\ud x \in \mathrm{Dist}_{x,\ud w}}
\mathcal{D}_{\ud x,\ud w},
\quad
\text{where }
\mathcal{D}_{\ud x,\ud w}
\cong
\mathbb G_{m,k}^{|J_{\ud x}^\circ|}
\times
\mathbb A_k^{|J_{\ud x}^-|}.
\end{align*}
Each Deodhar cell $\mathcal{D}_{\ud x,\ud w}$ is parametrized as follows:

\begin{align*}
\mathcal{D}_{\underline{x},\underline{w}} := \left\{
g\CI = g_1g_2\cdots g_r \CI \,\middle|\,
\begin{aligned}
g_j &= \dot{s}_{i_j}               && \text{if } j \in J_{\underline{x}}^+, \\
g_j &= \tilde y_{i_j}(p_j) && \text{if } j \in J_{\underline{x}}^\circ, \\
g_j &= \tilde x_{i_j}(m_j)\dot{s}_{i_j}^{-1} && \text{if } j \in J_{\underline{x}}^-.
\end{aligned}
\right\},
\end{align*}
where $p_j\in k^*$ and $m_j\in k$ are parameters.
\end{theorem}

We need the following lemma.
\begin{lemma}\label{lem:technical}
Let $w,w',y \in W_{\aff}$ with $w \le w'$. Assume that $\ell(w''y) = \ell(w'')+\ell(y)$ for any $w'' \in [w,w']$. Fix reduced expressions $\ud {w'}$ of $w'$ and $\ud y$ of $y$. Set $N=\ell(w')$ and $r=\ell(y)$. Then, the truncation map
\begin{align*}
\G_y: \mathrm{Dist}_{wy,\ud{w'}\ud y} \to \mathrm{Dist}_{w,\ud{w'}}, \quad
(x_{(0)},x_{(1)},\ldots,x_{(N)},\ldots,x_{(N+r)}) \mapsto (x_{(0)},x_{(1)},\ldots,x_{(N)})
\end{align*}
is a bijection.



\end{lemma}

\begin{proof}

Write $\ud y=s_{i_1}s_{i_2}\cdots s_{i_r}$. Let $
y_j=s_{i_1}s_{i_2}\cdots s_{i_j}$ for $0\leq j\leq r$, where \(y_0=1\) and \(y_r=y\). For every \(1\leq j\leq r\),
\begin{equation}\label{eq:technical-right-descent}
wy_js_{i_j}=wy_{j-1}<wy_j.
\end{equation}

Let $\ud{x} =(x_{(0)},x_{(1)},\ldots,x_{(N+r)})$
be a distinguished subexpression for \(wy\) in
\(\ud{w'}\ud y\). We claim that
\begin{equation}\label{eq:technical-forced-tail}
x_{(N+j)}=wy_j
\qquad (0\leq j\leq r).
\end{equation}
We prove this by descending induction on \(j\). The assertion for
\(j=r\) is clear. Suppose that \(j\geq 1\) and
\(x_{(N+j)}=wy_j\). By the definition of a subexpression, either $x_{(N+j)}=x_{(N+j-1)}$ or $x_{(N+j)}=x_{(N+j-1)}s_{i_j}$. The first possibility cannot occur. Indeed, distinguishedness would
then imply $wy_j=x_{(N+j)} = x_{(N+j-1)} \leq x_{(N+j-1)}s_{i_j} =wy_js_{i_j}$,
contradicting \eqref{eq:technical-right-descent}. Therefore, the
second possibility holds, and hence $x_{(N+j-1)}
=x_{(N+j)}s_{i_j}=wy_js_{i_j}
=wy_{j-1}$.
This proves \eqref{eq:technical-forced-tail}. In particular, \(x_{(N)}=w\). So $(x_{(0)},x_{(1)},\ldots,x_{(N)})$
is a distinguished subexpression for $w$ in $\ud{w'}$.

Let $(z_{(0)}, z_{(1)},\ldots, z_{(N)})$ be a distinguished subexpression for $w$ in $\ud{w'}$. By \eqref{eq:technical-right-descent}, $(z_{(0)}, z_{(1)},\ldots, z_{(N)}=w,wy_1,wy_2,\ldots,wy_r=wy)$ is a distinguished subexpression for $wy$ in $\ud{w'}\ud{y}$. This proves that the truncation map is bijective.
\end{proof}

\begin{proof}[Proof of Theorem \ref{thm:technical}]

By Lemma \ref{lem:length-additive}, right multiplication by $y$ induces an isomorphism of posets $[w,w']\cong [wy,w'y]$. Let $\pi_{w,w',y}$ be the natural projection
\begin{align*}
\bigsqcup_{x \in[wy,w'y]}\CI \dot{x} \CI / \CI
\cong
\bigl(\bigsqcup_{w'' \in[w,w']}\CI \dot{w}'' \CI\bigr)
\times^{\CI}
\CI \dot{y} \CI/ \CI
\longrightarrow
\bigsqcup_{w''\in[w,w']}\CI \dot{w}'' \CI / \CI.
\end{align*}

We claim that for any $w \le w_1 \le w_2 \le w'$,
\begin{align}\label{eq:3.3}
\pi_{w,w',y} \text{ induces a bijection }\widetilde{\CB}_{w_1y,w_2y}^{\circ} \longrightarrow \widetilde{\CB}_{w_1,w_2}^{\circ}.
\end{align}
Fix $w_1,w_2$ such that $w\leq w_1\leq w_2\leq w'$, a reduced expression $\ud{w_2}$ of $w_2$, and a reduced expression $\ud y$ of $y$. By Theorem \ref{thm:MR}, we have
\begin{align*}
\widetilde{\CB}_{w_1y,w_2y}^{\circ}
=
\bigsqcup_{\ud x \in \mathrm{Dist}_{w_1y,\ud{w_2}\ud y}}
\CD_{\ud x, \ud {w_2}\ud{y}},
\quad
\widetilde{\CB}_{w_1,w_2}^{\circ}
=
\bigsqcup_{\ud z \in \mathrm{Dist}_{w_1,\ud{w_2}}}
\CD_{\ud z, \ud {w_2}}.
\end{align*}
By Lemma \ref{lem:technical}, the truncation map $\G_y:\mathrm{Dist}_{w_1y,\ud{w_2}\ud y} \to \mathrm{Dist}_{w_1,\ud{w_2}}$ is a bijection. Using the explicit parameterization of these Deodhar pieces (cf. Theorem \ref{thm:MR}), we see that $\pi_{w,w',y}$ induces a bijection 
\begin{align*}
\CD_{\ud x, \ud {w_2}\ud{y}}
\xrightarrow{\sim}
\CD_{\G_y(\ud x), \ud {w_2}}.
\end{align*}
This proves the claim.

As a consequence, we see that $\pi_{w,w',y}$ induces a bijection $\widetilde{\CB}_{wy,w'y} \to \widetilde{\CB}_{w,w'}$. The target $\widetilde{\CB}_{w,w'}$ is seminormal by \cite[Proposition~6.6]{Kum17} in characteristic zero and by \cite[Proposition~5.3]{KS14} and \cite[Proposition~5.31]{HR76} in positive characteristic. Therefore, $\pi_{w,w',y}$ induces a $k$-isomorphism $\widetilde{\CB}_{wy,w'y} \xrightarrow{\sim} \widetilde{\CB}_{w,w'}$ which restricts to $k$-isomorphisms $\widetilde{\CB}_{w_1y,w_2y}^{\circ}\xrightarrow{\sim}\widetilde{\CB}_{w_1,w_2}^{\circ}$ for any $w\leq w_1\leq w_2\leq w'$.
\end{proof}

\subsection{Independence and geometric properties}
\label{sec:4.3}

We now compare the affine models and establish independence of the choices in their construction.

\begin{proposition}\label{prop:well-def}
Let $u,v\in W_0$. Let $x,x'\in W_0$ satisfy $u\leq_x v$ and $u\leq_{x'}v$, and let $\mu,\mu'$ be sufficiently dominant regular cocharacters. Then there is a $k$-isomorphism
\[
\CT_{u,v}^{\mu,x}\xrightarrow{\sim}\CT_{u,v}^{\mu',x'}
\]
which, for every $[u',v']_{\tilt}\subseteq [u,v]_{\tilt}$, restricts to a $k$-isomorphism
\[
\CT_{u',v'}^{\mu,x,\circ}\xrightarrow{\sim}
\CT_{u',v'}^{\mu',x',\circ}.
\]
\end{proposition}
\begin{proof}
Take sufficiently dominant regular cocharacters $\l,\l'$ such that $\mu+\l = \mu' + \l'$. By Proposition \ref{prop:realization}, we have isomorphisms of posets $[u, v]_{\tilt} \cong [w_{\mu,x}(u),w_{\mu,x}(v)]$ and $[u, v]_{\tilt} \cong [w_{\mu',x'}(u),w_{\mu',x'}(v)]$.

For any $z \in [u,v]_{\tilt}$, we have a length-additive product $w_{\mu,x}(z) w_{\l,u}^{\mathrm{III}}(x) = w_{\mu+\l, u}(z)$ and a length-additive product $w_{\mu',x'}(z) w_{\l',u}^{\mathrm{III}}(x') = w_{\mu'+\l', u}(z)$ (see \eqref{eq:other-realizations} for the type $\mathrm{III}$ realization). Indeed, this can be verified directly by the length formula \eqref{eq:len2}. By Theorem \ref{thm:technical}, we have isomorphisms of varieties $\widetilde{\CB}_{w_{\mu,x}(u),w_{\mu,x}(v)} \cong \widetilde{\CB}_{w_{\mu+\l,u}(u),w_{\mu+\l,u}(v)}$ and $\widetilde{\CB}_{w_{\mu',x'}(u),w_{\mu',x'}(v)} \cong \widetilde{\CB}_{w_{\mu'+\l',u}(u),w_{\mu'+\l',u}(v)}$. The composition gives rise to the desired $k$-isomorphism
\begin{align*}
\widetilde{\CB}_{w_{\mu,x}(u),w_{\mu,x}(v)} \cong \widetilde{\CB}_{w_{\mu',x'}(u),w_{\mu',x'}(v)}.    
\end{align*}\end{proof}

Our affine models have the following geometric properties.

\begin{theorem}\label{thm:tilted-Richardson}
Let $u,v \in W_0$. Let $x \in W_0$ be such that $u \le_x v$ and let $\mu$ be a sufficiently dominant regular cocharacter. We have
\begin{enumerate}
\item $\mathcal{T}_{u,v}^{\mu,x}$ is a projective closed
subvariety of $\widetilde{\CB}$, and
$\mathcal{T}_{u,v}^{\mu,x,\circ}\subseteq\mathcal{T}_{u,v}^{\mu,x}$
is an open subvariety.

\item We have a stratification
$$\mathcal{T}_{u,v}^{\mu,x} = \bigsqcup_{[u',v']_{\mathrm{tilt}} \subset [u,v]_{\mathrm{tilt}}} \mathcal{T}_{u',v'}^{\mu,x,\circ}.$$

\item The dimensions of $\mathcal{T}_{u,v}^{\mu,x}$ and $\mathcal{T}_{u,v}^{\mu,x,\circ}$ are equal to $\ell(u,v)$, the length of any shortest path from $u$ to $v$ in the quantum Bruhat graph.

\item Both $\mathcal{T}_{u,v}^{\mu,x}$ and $\mathcal{T}_{u,v}^{\mu,x,\circ}$ are irreducible varieties.

\end{enumerate}
\end{theorem}

\begin{proof}
Statements (1) and (4) follow from the properties of affine
Richardson varieties. We have
\begin{align*}
\CT_{u,v}^{\mu,x} &= \widetilde{\CB}_{w_{\mu,x}(u),w_{\mu,x}(v)}
 =\bigsqcup_{[w_1,w_2]  \subset [w_{\mu,x}(u),w_{\mu,x}(v)]} \widetilde{\CB}_{w_1,w_2}^{\circ}\\
 &= \bigsqcup_{ [u',v']_{\mathrm{tilt}} \subset [u,v]_{\mathrm{tilt}}} \widetilde{\CB}_{w_{\mu,x}(u'),w_{\mu,x}(v')}^{\circ}
= \bigsqcup_{ [u',v']_{\mathrm{tilt}} \subset [u,v]_{\mathrm{tilt}}} \CT_{u',v'}^{\mu,x,\circ}.
\end{align*}
The third equality follows from the poset isomorphism (Proposition \ref{prop:realization}). This proves (2). Statement (3) follows from Lemma \ref{lem:QBG2} (1) and \eqref{eq:len2}.
\end{proof}

\begin{definition}\label{def:abstract-tilted-richardson}
For $u,v\in W_0$, the \emph{abstract tilted Richardson variety}
$\CT_{u,v}^{\mathrm{abs}}$ is the $k$-isomorphism class of
$\CT_{u,v}^{\mu,x}$ equipped with its stratification labeled by the
subintervals of $[u,v]_{\tilt}$. The \emph{abstract open tilted
Richardson variety} $\CT_{u,v}^{\mathrm{abs},\circ}$ is the
$k$-isomorphism class of its open stratum $\CT_{u,v}^{\mu,x,\circ}$.
Proposition \ref{prop:well-def} shows that these classes are
independent of $\mu$ and $x$.
\end{definition}

Suppose $u\le v$ and take $x=1$. By Theorem~\ref{thm:technical},
we have $\CT_{u,v}^{\mu,1}
=\widetilde{\CB}_{ut^{-\mu},vt^{-\mu}}
\xrightarrow{\sim}\widetilde{\CB}_{u,v}$, compatible with every open stratum. Thus
$\CT_{u,v}^{\mathrm{abs}}$ is represented by
$\widetilde{\CB}_{u,v}$.

\begin{remark}
We can also use the type $\mathrm{III}$ affine realization $\Psi^{\mathrm{III}}_{\mu,x}$ to define the tilted Richardson varieties. All properties discussed above remain true. We omit the details.
\end{remark}

\section{\texorpdfstring{Tilted $R$-polynomials}{Tilted R-polynomials}}\label{sec:R-polynomials}

\subsection{Definition and point counting}

Let $\bq$ be an indeterminate. The {\it Kazhdan--Lusztig $R$-polynomials} are defined recursively using the following relations:

\begin{enumerate}
    \item $R_{u,v}(\bq) = 0$ if $u \not\leq v$,
    \item $R_{u,v}(\bq) = 1$ if $u = v$,
    \item If $vs < v$ for some simple reflection $s \in \BS_0$, then
    \[
    R_{u,v}(\bq) =
    \begin{cases}
        R_{us, vs}(\bq), & \text{if } us < u, \\
        \bq \cdot R_{us, vs}(\bq) + (\bq-1) \cdot R_{u, vs}(\bq), & \text{if } us > u.
    \end{cases}
    \]
\end{enumerate}
If $k = \BF_q$ for some prime power $q$, Deodhar \cite{Deod85} proved that when $u\le v$, we have
\begin{align*}
    R_{u,v}(q) = \#\mathring{\CB}_{u,v}(k).
\end{align*}

For any $w,w' \in W_{\aff}$, define the affine $R$-polynomial $\widetilde{R}_{w,w'}(\bq)$ in the same way as in the finite case.

\begin{definition}\label{def:tilted-R-poly}
Let $u,v \in W_0$. Let $x \in W_0$ be such that $u \le_x v$ and let $\mu \in Q^{\vee}$ be a sufficiently dominant regular cocharacter. Consider the affine realization $w_{\mu,x}(z) = z t^{-(\mu + \wt(x,z))} x^{-1}$ for $z \in W_0$ (cf. \eqref{eq:psi}). The \emph{tilted $R$-polynomial} is defined as
\begin{align*}
    R_{u,v}^{\mu,x,\mathrm{tilt}}(\bq) = \tilde{R}_{w_{\mu,x}(u),w_{\mu,x}(v)}(\bq).
\end{align*}
\end{definition}

By the Deodhar decomposition and Deodhar's formula for affine $R$-polynomials, we have 
\begin{proposition}\label{prop:point-counting}
Let $q$ be a prime power and take $k=\mathbb F_q$.
For any $x\in W_0$ such that $u\le_x v$ and any sufficiently
dominant regular cocharacter $\mu$, we have
\[
R_{u,v}^{\mu,x,\mathrm{tilt}}(q)
=\#\CT_{u,v}^{\mu,x,\circ}(\mathbb F_q).
\]
\end{proposition}

As a corollary to Proposition~\ref{prop:well-def}, the polynomial
$R_{u,v}^{\mu,x,\mathrm{tilt}}(\bq)$ is independent of the choices of $\mu$ and $x$. We denote the common value by
$R_{u,v}^{\mathrm{tilt}}(\bq)$. In particular, we have $R_{u,v}^{\mathrm{tilt}}(q) =\#\CT_{u,v}^{\mathrm{abs},\circ}(\mathbb F_q)$ for any prime power $q$.

\begin{remark}
One may also use the type $\mathrm{III}$ affine realization in \eqref{eq:other-realizations} to define the tilted $R$-polynomials. We omit the details.
\end{remark}

\subsection{Minimal weight paths and the path formula}

We first introduce the tilted Bruhat graph and the minimal weight paths.

Recall that the set of affine roots is $\Phi_{\aff} = \{\a + m\d \mid \a \in \Phi,m\in\BZ\}$. By convention, the set of positive affine roots is $\Phi^+_{\aff} = \{\a+m\d \in \Phi_{\aff} \mid m\ge \d_{\Phi^-}(\a)\}$. Here, for any $\b\in\Phi$, we denote
$$\d_{\Phi^-}(\b) = \begin{cases}
    1 & \b \in \Phi^-\\
    0 & \b \in \Phi^+.
\end{cases}$$

\begin{definition}
The tilted Bruhat graph $\TBG(\Phi)$ is a directed graph defined as follows, extending the quantum Bruhat graph.
\begin{enumerate}
\item Vertices: Elements of $W_0$.
\item Edges: 
\begin{enumerate}[(i)]
    \item Bruhat edges: $w \to w s_\alpha$ when $w s_\alpha > w$, $\a \in \Phi^+$.
    \item Quantum edges: $w \to w s_\alpha$ when $w s_\a < w$, $\a\in\Phi^+$.
\end{enumerate}
\item Weights: A Bruhat edge $w \to w s_\alpha$ has weight 0; a quantum edge $w \to w s_\alpha$ has weight $\alpha^\vee$.
\item Let $p: u_0 = u \xrightarrow{\b_1}u_1\xrightarrow{\b_2}\cdots \xrightarrow{\b_r}u_r=v$ with $u_{i+1} = u_i s_{\b_{i+1}}$ be a path in $\TBG(\Phi)$. The length of $p$ is the number of its edges, denoted by $\ell(p)=r$. The \emph{weight} of $p$ is the sum of the weights of its edges, denoted by $\wt(p)$. In other words, $\wt(p) = \sum_{i=1}^r \d_{\Phi^-}(u_{i-1}(\b_i))\b_i^{\vee}$.

\end{enumerate}
    
\end{definition}

Note that paths in $\QBG(\Phi)$ are naturally viewed as paths in $\TBG(\Phi)$.

\begin{definition}
We say a path $p$ in $\TBG(\Phi)$ from $u$ to $v$ is of \emph{minimal weight} if $\wt(p)=\wt_{\QBG}(u,v)$. 
\end{definition}

\begin{definition}[{\cite[Definition 2.1]{Dyer93}}]\label{def:reflection-order}
\begin{enumerate}
\item 
A total order $\preceq$ on $\Phi^+$ is called a {\it reflection order} if for any $\a, \b \in \Phi^+$ and $a, b > 0$ such that $\a \prec \beta$ and $a\a + b\b \in \Phi^+$, we have $\a \prec a\a + b\b \prec \b$.

\item Replacing $\Phi^+$ with $\Phi_{\aff}^+$, we define the notion of \emph{reflection order} on $\Phi_{\aff}^{+}$. 
\end{enumerate}
\end{definition} 
Let $\preceq$ be a reflection order on $\Phi^+$. A path $p: u_0 = u \xrightarrow{\b_1}u_1\xrightarrow{\b_2}\cdots \xrightarrow{\b_r}u_r=v$ in $\TBG(\Phi)$ is called \emph{label-increasing} with respect to $\preceq$ if $\b_1\preceq \b_2\preceq \cdots \preceq \b_r$.

Note that all shortest paths in $\QBG(\Phi)$, when viewed as paths in $\TBG(\Phi)$, are of minimal weight. However, not all minimal weight paths lie in $\QBG(\Phi)$.
\begin{example}
In the case of type $A_2$, choose the reflection order $\a_1\prec \a_{1} + \a_2\prec\a_2$. The path $1\to s_2s_1s_2 \to s_2 s_1$ in $\TBG(\Phi)$ is label-increasing but not of minimal weight. The path $s_2 s_1\xrightarrow{\a_1+\a_2} s_1$ is a minimal weight edge in $\TBG(\Phi)$, but it is not an edge in $\QBG(\Phi)$. Indeed, we have $\ell_{\QBG}(s_2 s_1,s_1) = 3$ (see Figure \ref{fig:1}).
\end{example}

The notion of the tilted Bruhat graph we use here is a degenerate version of the double Bruhat graph in \cite{NW17} and \cite{Sch25a}. The relation between $\QBG(\Phi)$ and $\TBG(\Phi)$ is similar to the relation between the Hasse diagram of the Bruhat order and the (full) Bruhat graph.

\needspace{6\baselineskip}

The main result of this section is the following.

\begin{theorem}\label{thm:R-poly}
Let $\preceq$ be a reflection order on $\Phi^+$. Let $u,v\in W_0$. Then 
\begin{align}\label{eq:R-poly}
R_{u,v}^{\tilt}(\bq)
= \bq^{\frac{\ell(u,v)}{2}}
\sum_{p}
\left(\bq^{\frac{1}{2}} - \bq^{-\frac{1}{2}}\right)^{\ell(p)},
\end{align}
where $p$ runs over all label-increasing paths (with respect to $\preceq$) in $\TBG(\Phi)$ from $u$ to $v$ of minimal weight.    
\end{theorem}
\begin{remark}
The same formula for $R_{u,v}^{\mathrm{tilt, GGG}}(\bq)$ in the type $A$ case was conjectured in \cite[Conjecture 6.23]{GGG26}.

\end{remark}

\subsection{Proof via affine Bruhat chains}

We first give a simple lemma on minimal weight paths.

\begin{lemma}\label{lem:TBG}
Let $u,v\in W_0$. Let $p:u = u_0 \to u_1\to \cdots \to u_r = v$ be a path from $u$ to $v$ in $\TBG(\Phi)$.
\begin{enumerate}
    \item We have $\wt(p)\ge \wt_{\QBG}(u,v)$.

    \item The path $p$ is a minimal weight path if and only if, for each $j$, (i) the edge $u_{j-1}\to u_j$ is a minimal weight edge, (ii) $u_j$ lies in $[u,v]_{\tilt}$ and (iii) $u_{j-1} \le_{u,v} u_j$ (see Definition \ref{def:tilted-bruhat-interval} for the notation).
\end{enumerate}
\end{lemma}

\begin{proof}
By \cite[Lemma 4.3]{Sch23}, for $z \in W_0$ and $\b \in \Phi^+$, we have
\begin{align}\label{eq:TBG-edge-bound}
\wt_{\QBG}(z, zs_\beta)\leq \d_{\Phi^-}(z(\b))\b^{\vee}.
\end{align}
We now prove the two statements.

\begin{enumerate}

\item Applying Lemma \ref{lem:QBG2} (3) and \eqref{eq:TBG-edge-bound}, we obtain
\[\wt_{\QBG}(u,v) \leq \sum_{j=1}^r\wt_{\QBG}(u_{j-1},u_j) \le \wt(p).\]

\item
Applying Lemma \ref{lem:QBG2} (3) and \eqref{eq:TBG-edge-bound}, we get
\begin{align*}
\wt_{\QBG}(u,v) \leq \sum_{j=1}^r\wt_{\QBG}(u_{j-1},u_j) \le \wt(p).
\end{align*}
Then $\wt(p) = \wt_{\QBG}(u,v)$ if and only if all equalities hold. All statements follow.
\end{enumerate}
\end{proof}

By convention, for an affine root $\tg = \g + m\d$, the corresponding affine reflection is defined as
\begin{align}\label{eq:stg}
s_{\tg} = t^{m\g^{\vee}}s_{\g} = s_{\g} t^{- m\g^{\vee}}.
\end{align}

\begin{definition}
An \emph{(upward) Bruhat edge} in $W_{\aff}$ is of the form $w\xrightarrow{\tg}ws_{\tg}$, $\tg\in\Phi^+_{\aff}$ such that $w<ws_{\tg}$. An \emph{(upward) Bruhat chain} in $W_{\aff}$ is of the form
\begin{align*}
w \xrightarrow{\tg_1} w_1\xrightarrow{\tg_2} w_2\xrightarrow{\tg_3}\cdots\xrightarrow{\tg_r}w_r = w',
\end{align*}
where each edge is an upward Bruhat edge. We call these $\tg_i$ the \emph{(right) labels} of the chain. Let $\preceq_{\aff}$ be a reflection order on $\Phi_{\aff}^+$ (see Definition \ref{def:reflection-order}). The chain is called \emph{(right) label-increasing} if $\tg_1\preceq_{\aff}\tg_2 \preceq_{\aff}\cdots\preceq_{\aff}\tg_{r}$.
\end{definition}

The affine realization identifies minimal weight paths with affine Bruhat chains, as follows.

\begin{proposition}\label{prop:path-corr} 
Let $x \in W_0$ and let $\mu$ be a sufficiently dominant regular cocharacter. Let $u,v\in W_0$ with $u\le_x v$.
\begin{enumerate}
\item There is a minimal weight edge $u \xrightarrow{\b} v$ in $\TBG(\Phi)$ if and only if $w_{\mu,x}(u) = w_{\mu,x}(v)s_{\tg}$ for some $\tg\in\Phi_{\aff}^+$ (see \eqref{eq:stg} for the notation). In this case, we have $\tg = x(\b) + (\d_{\Phi^-}(u(\b)) + \<\mu+\wt_{\QBG}(x,u),\b\>)\d$.  

\item The affine realization $\Psi_{\mu,x}$ induces a bijection between the set of minimal weight paths from $u$ to $v$ in $\TBG(\Phi)$ and the set of upward Bruhat chains from $w_{\mu,x}(u)$ to $w_{\mu,x}(v)$ in $W_{\aff}$.
\end{enumerate}
\end{proposition}

\begin{proof}
\begin{enumerate}
\item If $u^{-1}v$ is not a reflection, then both statements fail. Therefore, we may assume that $v = us_\b$ for some $\b \in \Phi^+$. Direct computation gives
\[
w_{\mu,x}(v)
=w_{\mu,x}(u)s_{x(\b)}
t^{-\<\mu+\wt_{\QBG}(x,u),\b\>x(\b^{\vee})
-x(\wt_{\QBG}(u,v))}.
\]
Then
\begin{align*}
&\quad \quad \quad w_{\mu,x}(v) =  w_{\mu,x}(u)s_{\tg} \text{ for some }\tg \in \Phi^+_{\aff} \\
&\iff \wt_{\QBG}(u,v) = c \b^{\vee} \text{ for some } c \in \BZ_{\ge0}\\
&\iff \wt_{\QBG}(u,v) = \d_{\Phi^-}(u(\b))\b^{\vee}\\
& \iff u\xrightarrow{\b} v\text{ is a minimal weight edge in }\TBG(\Phi).
\end{align*}
Here, the second equivalence follows from the fact that $\wt_{\QBG}(u,v) \le \d_{\Phi^-}(u(\b))\b^{\vee}$ (cf. \eqref{eq:TBG-edge-bound}) and Lemma \ref{lem:QBG2}(2). This proves (1).

\item By (1) and Lemma \ref{lem:TBG}, any minimal weight path from $u$ to $v$ in $\TBG$ corresponds to an upward Bruhat chain from $w_{\mu,x}(u)$ to $w_{\mu,x}(v)$. Conversely, by Proposition \ref{prop:realization} and Lemma \ref{lem:TBG}, any upward Bruhat chain from $w_{\mu,x}(u)$ to $w_{\mu,x}(v)$ is of this form. \qedhere
\end{enumerate}
\end{proof}

The following formula for $R$-polynomials is due to Dyer (with a different normalization). 
\begin{proposition}
[{\cite[Proposition 3.3]{Dyer93} and \cite[Corollary 3.4]{Dyer93}}]\label{thm:dyer}

Let $\preceq_{\aff}$ be a reflection order on $\Phi^+_{\aff}$. Let $w,w'\in W_{\aff}$. Then
\begin{align*}
\widetilde{R}_{w,w'}(\bq)
= \bq^{\frac{\ell(w') - \ell(w)}{2}}
\sum_{p}
\left(\bq^{\frac{1}{2}} - \bq^{-\frac{1}{2}}\right)^{\ell(p)},
\end{align*}
where $p$ runs over all (right) label-increasing (with respect to $\preceq_{\aff}$) upward Bruhat chains from $w$ to $w'$.
\end{proposition}

By the general theory of the double Bruhat graph and the quantum Yang--Baxter equation, the right-hand side of equation \eqref{eq:R-poly} is independent of the choice of the reflection order (cf. \cite[Corollary 4.9]{Sch25a} and \cite[Theorem 3.1]{BFP}). Hence, it suffices to prove Theorem \ref{thm:R-poly} for one specifically chosen reflection order $\preceq$. 

The key to the proof of Theorem \ref{thm:R-poly} is to specifically construct a reflection order $\preceq$ on $\Phi^+$ and a reflection order $\preceq_{\aff}$ on $\Phi^+_{\aff}$ that are ``compatible'' in the sense that the correspondence induced by affine realization (cf. Proposition \ref{prop:path-corr}) preserves the label-increasing properties.

We construct the reflection orders as follows.

Let $\mathfrak{h}^* = \BZ\Phi_{\aff} \otimes_{\BZ}\BR$ and $\mathfrak{h}^*_0 = \BZ\Phi \otimes_{\BZ}\BR$. Then $\fkh^* = \fkh^*_0 \oplus \BR\d$. Note that $W_0$ acts on $\fkh^*$ and stabilizes $\d$ and $\fkh^*_0$. Recall that $\tilde I = I\sqcup \{0\}$, $\ta_i = \a_i$ for $i\in I$ and $\ta_0 = -\th+\d$. Let 

\begin{gather*} \CU = \left\{\sum_{i \in I} c_i \a_i \mid 0\le c_i \le 1, \sum_{i \in I}  c_i = 1\right\} \subset \fkh^*_0, \\ \tilde \CU = \left\{\sum_{i \in \tilde I} c_i \ta_i \mid 0\le c_i \le 1, \sum_{i \in \tilde I}  c_i = 1\right\} \subset \fkh^*.\end{gather*}
Note that the simplex $\CU$ is a face of the simplex $\tilde\CU$. For any $\b\in\Phi$, define $(c_{\b,i})_{i \in I}$ via $\b = \sum_{i \in I} c_{\b,i} \a_i$. Define $\mathrm{ht}(\b) = \sum_{i \in I}c_{\b,i} \in \BZ$. For any $\tg\in\Phi_{\aff}$, define $(\tilde{c}_{\tg,i})_{i \in \tilde I}$ via $\tg = \sum_{i \in \tilde I} \tilde{c}_{\tg,i} \ta_i$. Define $\widetilde{\mathrm{ht}}(\tg) = \sum_{i \in \tilde I}\tilde{c}_{\tg,i} \in \BZ$. We denote $\pi(\b)=  \frac{\b}{\mathrm{ht}(\b)} \in \CU$ and $\pi(\tg) = \frac{\tg}{\widetilde{\mathrm{ht}}(\tg)} \in \tilde{\CU}$.

Choose a $W_0$-invariant inner product $(-,-)$ on $\fkh^* =  \fkh^*_0 \oplus \BR\d$ such that $\fkh^*_0$ and $\BR\d$ are orthogonal. Choose a vector $h\in \fkh^*_0$ such that the numbers $\left(h,\pi(\b)\right), \b \in \Phi^+$ are distinct and non-zero. Define a total order $\preceq$ on $\Phi^+$ via
\begin{align}\label{eq:prec}
    \b \prec \b' \iff \left(h, \pi(\b)\right) < \left(h, \pi(\b')\right).
\end{align}
This is clearly a reflection order on $\Phi^+$. We call the vector $h \in \fkh^*_0$ a \emph{direction vector} of $\preceq$.

Let $x \in W_0$. We can define a reflection order $\preceq_{x, \aff}$ on $\Phi_{\aff}^+$ such that for any $\tg,\tg' \in \Phi^+_{\aff}$, 
\begin{align}\label{eq:precaff}
\left(x(h), \pi(\tg)\right) < \left(x(h), \pi(\tg')\right) \implies \tg \prec_{x, \aff} \tg'.    
\end{align}
Indeed, inside the orthogonal complement of $x(h)$ in $\fkh^*$, we choose inductively a direction vector $h_1$, a direction vector $h_2$ in its orthogonal complement, and so on (this is a standard way to construct reflection orders, cf. \cite[proof of Proposition 2.3]{Dyer93}). We call $x(h) \in \fkh^*_0$ a \emph{direction vector} of $\preceq_{x,\aff}$.

Figure \ref{fig:4} illustrates the idea of our construction in the case of $C_2$, where $\CU$ is the thick red segment and $\tilde\CU$ is the light red triangle. We choose $x = s_1s_2s_1$. The direction vectors $h$ and $x(h)$ are drawn explicitly. For a sufficiently large $N$, points $\pi(\b_j)$ in $\CU$ correspond to points $\pi(\tg_j)\in\tilde\CU$ that are very close to the thin black line.

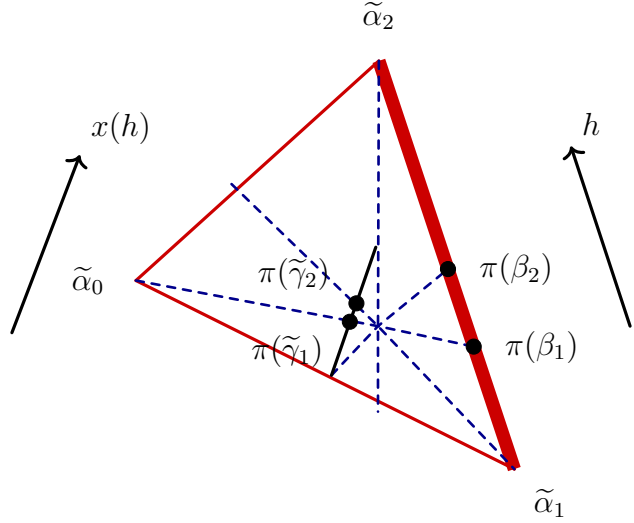
\begin{figure}[htbp]
  \centering
  \begin{tikzpicture}[
    scale=.58,
    line cap=round,
    line join=round,
    every node/.style={font=\large},
    auxiliary/.style={
      blue!55!black,dashed,line width=1pt
    },
    thin side/.style={
      red!80!black,line width=1.2pt
    },
    thick side/.style={
      red!80!black,line width=5pt,line cap=butt
    },
    point/.style={
      circle,draw=black,fill=black,inner sep=2.0pt
    }
  ]
    \coordinate (O)  at ( 0.000, 0.000);
    \coordinate (A0) at (-5.520, 1.040);
    \coordinate (A2) at ( 0.020, 6.050);
    \coordinate (Ai) at ( 3.110,-3.260);

    \draw[thin side]  (A0)--(A2)--(Ai);
    \draw[thin side]  (Ai)--(A0);
    \draw[thick side] (A2)--(Ai);

    \coordinate (P)  at (-3.340, 3.240);
    \coordinate (Q)  at (-1.060,-1.120);
    \coordinate (V)  at ( 0.000,-1.950);
    \coordinate (B1) at ( 2.183,-0.467);
    \coordinate (B2) at ( 1.596, 1.302);

    \draw[auxiliary] (A0)--(O)--(B1);
    \draw[auxiliary] (A2)--(O)--(V);
    \draw[auxiliary] (P)--(O)--(Ai);
    \draw[auxiliary] (Q)--(O)--(B2);

    \coordinate (R)  at (-0.050,1.800);
    \coordinate (G1) at (-0.642,0.100);
    \coordinate (G2) at (-0.492,0.520);

    \draw[black,line width=1.2pt] (Q)--(R);

    \node[point] (g1) at (G1) {};
    \node[point] (g2) at (G2) {};
    \node[point] (b1) at (B1) {};
    \node[point] (b2) at (B2) {};

    \node[below left=-1pt and 4pt of g1]
      {$\pi(\widetilde\gamma_1)$};
    \node[above left=-1pt and 4pt of g2]
      {$\pi(\widetilde\gamma_2)$};
    \node[right=5pt of b1]
      {$\pi(\beta_1)$};
    \node[right=5pt of b2]
      {$\pi(\beta_2)$};

    \node[left=8pt] at (A0)
      {$\widetilde\alpha_0$};
    \node[above=8pt] at (A2)
      {$\widetilde\alpha_2$};
    \node[below right=5pt] at (Ai)
      {$\widetilde\alpha_1$};

    \draw[->,line width=1.2pt]
      (-8.350,-0.150)--(-6.800,3.900)
      node[above right] {$x(h)$};

    \draw[->,line width=1.2pt]
      (5.750,0.000)--(4.400,4.100)
      node[above right] {$h$};
  \end{tikzpicture}

  \caption{$C_2$ example for $\preceq$ and $\preceq_{x,\aff}$
    with $x=s_1s_2s_1$.}
  \label{fig:4}
\end{figure}

The next proposition shows that these reflection orders are compatible with the affine realization.

\begin{proposition}\label{prop:label-increasing}
Let $x \in W_0$. Let $\preceq$ and $\preceq_{x,\aff}$ be as above. Let $\rho^{\vee}$ be the half-sum of positive coroots. Then, for any sufficiently large even number $N$, the affine realization $\Psi_{N\rho^{\vee},x}$ induces a bijection between the set of label-increasing paths (with respect to $\preceq$) of minimal weight from $u$ to $v$ in $\TBG(\Phi)$ and the set of label-increasing upward Bruhat chains (with respect to $\preceq_{x,\aff}$) from $w_{N\rho^{\vee},x}(u)$ to $w_{N\rho^{\vee},x}(v)$ in $W_{\aff}$, for any $u,v \in W_0$ such that $u\le_x v$.  
\end{proposition}

\begin{proof}
Let $N$ be a large even number such that $N\rho^{\vee} \in Q^{\vee}$ is a sufficiently dominant regular cocharacter. Let $u,v\in W_0$ satisfy $u\le_x v$, and let
\[
p:
u_0=u\xrightarrow{\b_1}u_1
\xrightarrow{\b_2}\cdots
\xrightarrow{\b_r}u_r=v
\]
be a minimal weight path in $\TBG(\Phi)$. By Proposition \ref{prop:path-corr}(2), the affine realization
$\Psi_{N\rho^\vee,x}$ sends $p$ to the upward Bruhat chain
\[
\begin{aligned}
w_{N\rho^\vee,x}(u)
=w_{N\rho^\vee,x}(u_0)
&\xrightarrow{\tg_{1,N}}w_{N\rho^\vee,x}(u_1)
\xrightarrow{\tg_{2,N}}\cdots \xrightarrow{\tg_{r,N}}w_{N\rho^\vee,x}(u_r)
=w_{N\rho^\vee,x}(v).
\end{aligned}
\]
By the label formula in Proposition \ref{prop:path-corr}(1), the
$j$-th label of this chain is
\begin{align*}
\tg_{j, N}
&=
x(\b_j)
+
\left(
\d_{\Phi^-}(u_{j-1}(\b_j))
+\<N\rho^\vee+\wt(x,u_{j-1}),\b_j\>
\right)\d.
\end{align*}
Note that $\<\rho^{\vee},\b_j\> = \mathrm{ht}(\b_j)$. Denote $\widetilde{\mathrm{ht}}(\d) := \widetilde{\mathrm{ht}}(\ta_0) + \widetilde{\mathrm{ht}}(\th)= 1 + \mathrm{ht}(\th)$. Then $\widetilde{\mathrm{ht}}(\tg_{j,N}) = N\mathrm{ht}(\b_j)\widetilde{\mathrm{ht}}(\d)+O(1)$ as $N \to \infty$.

Since $\fkh^*_0$ and $\BR\d$ are orthogonal, and the inner product is
$W_0$-invariant, the above equation implies
\begin{align*}
\lim_{N\to\infty} N\widetilde{\mathrm{ht}}(\d)
\left(x(h),\pi(\tg_{j,N})\right) &=
\lim_{N\to\infty} \frac{N\widetilde{\mathrm{ht}}(\d)(x(h),x(\b_j))}
{\widetilde{\mathrm{ht}}(\tg_{j,N})}\\
&=\lim_{N\to\infty} 
\frac{N \widetilde{\mathrm{ht}}(\d)(h,\b_j)}
{N \mathrm{ht}(\b_j)\widetilde{\mathrm{ht}}(\d) + O(1)}\\
&=
\frac{(h,\b_j)}{\mathrm{ht}(\b_j)}
=
\left(h,\pi(\b_j)\right).
\end{align*}

Notice that $\b_{j-1}\ne\b_j$ for $2\le j\le r$. Indeed, two consecutive edges with the
same label form a two-edge loop of weight $\b_j^\vee$, contradicting the assumption that $p$ is a minimal weight path. By the choice of $h$, the numbers $\left(h,\pi(\b)\right), \b\in\Phi^+$ are pairwise distinct. Hence, for all sufficiently large $N$,
\[
\left(h,\pi(\b_{j-1})\right)
<
\left(h,\pi(\b_j)\right)
\quad\Longleftrightarrow\quad
\left(x(h),\pi(\tg_{j-1,N})\right)
<
\left(x(h),\pi(\tg_{j,N})\right).
\]
Therefore, by \eqref{eq:prec} and \eqref{eq:precaff}, we have
\[
\b_{j-1}\prec\b_j
\quad\Longleftrightarrow\quad
\tg_{j-1,N}\prec_{x,\aff}\tg_{j,N}.
\]
It follows that $p$ is label-increasing if and only if its corresponding upward
Bruhat chain is label-increasing.

We choose one sufficiently large even number $N$ such that the above estimates work uniformly for all $u,v$ with $u\le_x v$ and all minimal weight paths from $u$ to $v$. Then the bijection in Proposition \ref{prop:path-corr}(2) restricts
to the desired bijection.
\end{proof}

Now Theorem \ref{thm:R-poly} follows from Proposition \ref{thm:dyer} and Proposition~\ref{prop:label-increasing}.

\subsection{Comparison with the type $A$ construction}

We conjecture that the tilted $R$-polynomials agree with those
arising from the type $A$ construction of Gao--Gao--Gao.

\begin{conjecture}\label{conj:R-polynomial-comparison}
In type $A$, the tilted $R$-polynomial
$R_{u,v}^{\mathrm{tilt}}(\bq)$ coincides with the polynomial
$R_{u,v}^{\mathrm{tilt,GGG}}(\bq)$ of Gao--Gao--Gao
\cite{GGG26}.
\end{conjecture}


\begin{example} Assume $G = \GL_6$. We have $W_0 = S_6$, the symmetric group on $6$ elements. Let $v = 246513$ and $u = 512346$. This is the case considered in \cite[Example 6.4]{GGG26}. By the explicit calculation therein, there are four distinguished subexpressions. By \cite[Proposition 6.18]{GGG26}, we have 
\begin{align*}
R_{u,v}^{\mathrm{tilt,GGG}}(\bq) = (\bq-1)^8 + 2(\bq-1)^6 \bq + (\bq-1)^4 \bq^2.    
\end{align*}
Now take $x = u$ and $\mu =[120,72,24,-24,-72,-120]$. We have $\ell(u,v)=8$ and $\wt(u,v) = [1,0,-1,0,0,0]$. 
By computer calculation (using the recurrence relation), we get  
\begin{align*}
\tilde{R}_{w_{\mu,x}(u),w_{\mu,x}(v)}(\bq)= 1 - 6 \bq + 17 \bq^2 - 30 \bq^3 + 36 \bq^4 - 30 \bq^5 + 17 \bq^6 -
6 \bq^7 + \bq^8.
\end{align*}
This coincides with $R_{u,v}^{\mathrm{tilt,GGG}}(\bq)$.
\end{example}

\section{Total positivity and geometric realization}
\label{sec:TNN}\label{sec:tnn-flag}

Throughout this section, we take $k=\BR$.

We restrict our attention to the real points $G(\BR)$ and $\mathcal B(\BR)$. Define $U^-_{\geq0}\subset U^-(\BR)$ as the monoid generated by $y_i(a)$ for all $i\in I$ and $a\in\BR_{>0}$. The \emph{totally nonnegative flag variety} is defined as the Hausdorff closure:
\[
\mathcal B_{\geq0}
:=
\overline{U^-_{\geq0}\cdot B/B}
\subset\mathcal B(\BR).
\]

For any pair $u, v \in W_0$ with $u \le v$, the \emph{totally positive Richardson cell} is defined as:
\[ \mathring{\mathcal{B}}_{u,v, >0} := \mathring{\mathcal{B}}_{u,v}(\mathbb{R}) \cap \mathcal{B}_{\ge 0}. \] Then we have $\mathcal{B}_{\ge 0} = \bigsqcup_{u \le v} \mathring{\mathcal{B}}_{u, v, >0}$.
The topological closure of this cell inside $\mathcal{B}_{\ge 0}$ is the \emph{totally nonnegative Richardson variety}, denoted by $\mathcal{B}_{u,v, \ge 0}$. The following algebraic and topological properties were established by Lusztig \cite{Lusz94} and Rietsch \cite{Riet99} and \cite{Riet06}. 

\begin{theorem}
Let $u, v \in W_0$ with $u \le v$. Then 
\begin{enumerate}
    \item The totally positive Richardson cell $\mathring{\mathcal{B}}_{u,v, >0}$ is a connected component of $\mathring{\mathcal{B}}_{u, v}(\BR)$.
    \item The totally positive Richardson cell $\mathring{\mathcal{B}}_{u,v, >0}$ is homeomorphic to $\mathbb R_{>0}^{\ell(v)-\ell(u)}$.
    
    \item The totally nonnegative Richardson variety satisfies
    \[ \mathcal{B}_{u,v, \ge 0}=\bigsqcup_{u \le x \le y \le v} \mathring{\mathcal{B}}_{x,y, >0}. \]
\end{enumerate}
\end{theorem}

The following result was conjectured by Postnikov and Williams in \cite{LW} and was established by Galashin--Karp--Lam in \cite{GKL22}.

\begin{theorem}
For any $u \le v$ in $W_0$, the totally nonnegative Richardson variety $\mathcal{B}_{u,v, \ge 0}$ forms a regular CW-complex that is homeomorphic to a closed real ball. The face poset of this regular CW-complex is canonically isomorphic to the poset of subintervals in $[u,v]$.
\end{theorem}

We now pass to the affine flag variety.

For $i \in \tilde{I}$, let $\tilde x_i$, $\tilde y_i$ be as in \S\ref{sec:4.1}. Let $\widetilde U^-_{\geq 0}\subset LG(\BR)$ be the monoid generated
by $\widetilde y_i(a)$ for $i\in\tilde I$ and $a\in\BR_{>0}$. Define
the \emph{totally nonnegative part} of the affine flag variety by
\begin{align*}
\widetilde{\CB}_{\geq 0}
=
\overline{\widetilde U^-_{\geq 0}\cdot\CI/\CI}
\subseteq \widetilde{\CB}(\BR),
\end{align*}
where the closure is taken in the Hausdorff topology.

For $w\leq w'$ in $W_{\aff}$, define
\begin{align*}
\widetilde{\CB}_{w,w',>0}=
\widetilde{\CB}_{w,w'}^\circ(\BR)
\cap\widetilde{\CB}_{\geq 0},\quad
\widetilde{\CB}_{w,w',\geq 0}
=
\overline{\widetilde{\CB}_{w,w',>0}}.
\end{align*}

Total positivity for affine flag varieties is substantially subtler than in finite type, as the duality between opposite Borel subgroups used in the reductive case is unavailable in the Kac--Moody setting. The joint work of Bao and the first-named author \cite{BH24} established the following result by developing a product structure for totally positive Richardson varieties, enabling an inductive study of their closures and local topology.

\begin{theorem}[{\cite[Proposition~5.2 and Theorem~5.3]{BH24}}]
\label{thm:affine-TNN}
Let $w\leq w'$ in $W_{\aff}$. Then for each $w \le w_1\le w_2 \le w'$, $\widetilde{\CB}_{w_1,w_2,>0}$ is homeomorphic to
$\BR_{>0}^{\ell(w_2)-\ell(w_1)}$, and
\begin{align*}
\widetilde{\CB}_{w,w',\geq 0}
=
\bigsqcup_{w\leq w_1\leq w_2\leq w'}
\widetilde{\CB}_{w_1,w_2,>0}.
\end{align*}
This is a regular CW-complex homeomorphic to a closed ball of
dimension $\ell(w')-\ell(w)$. Its face poset is the poset of
subintervals of $[w,w']$, ordered by inclusion.
\end{theorem}

We next show that the isomorphism in Theorem
\ref{thm:technical} is compatible with total nonnegativity.

\begin{proposition}\label{prop:TNN-compatible}
Let $w,w',y\in W_{\aff}$ satisfy the assumptions of Theorem
\ref{thm:technical}. Let $\pi_{w,w',y}$ be the projection as in \eqref{eq:3.3}. For each $w\leq w_1\leq w_2\leq w'$, $\pi_{w,w',y}$ restricts to homeomorphisms
\begin{align*}
\widetilde{\CB}_{w_1y,w_2y,>0}
\xrightarrow{\sim}
\widetilde{\CB}_{w_1,w_2,>0},\quad
\widetilde{\CB}_{w_1y,w_2y,\geq 0}
\xrightarrow{\sim}
\widetilde{\CB}_{w_1,w_2,\geq 0}.
\end{align*}
\end{proposition}

\begin{proof}
The statement follows from the proof of Theorem \ref{thm:technical} and the parameterization of Richardson cells via positive distinguished subexpressions (cf. \cite[Theorem 4.10]{BH21}).
\end{proof}

Let $u,v\in W_0$. Choose $x\in W_0$ such that $u\leq_x v$ and
choose a sufficiently dominant regular cocharacter $\mu$. Define
\begin{align*}
\CT_{u,v}^{\mu,x,>0}=\widetilde{\CB}_{w_{\mu,x}(u),w_{\mu,x}(v),>0}, \quad \CT_{u,v}^{\mu,x,\geq 0}=\widetilde{\CB}_{w_{\mu,x}(u),w_{\mu,x}(v),\geq 0}.
\end{align*}
The isomorphism constructed in Proposition \ref{prop:well-def} is
a composition of the isomorphisms in Theorem \ref{thm:technical}
and their inverses. Hence, Proposition \ref{prop:TNN-compatible}
shows that it restricts to a cell-preserving homeomorphism between
the corresponding totally nonnegative parts. We therefore write
$\CT_{u,v}^{\mathrm{abs},>0}$ and $\CT_{u,v}^{\mathrm{abs},\geq 0}$ for their homeomorphism
types.

Combining Proposition \ref{prop:realization} and Theorem \ref{thm:affine-TNN}, we obtain the following.
\begin{theorem}\label{thm:TNN}
Let $u,v \in W_0$. Let $x \in W_0$ be such that $u \le_x v$ and let $\mu$ be a sufficiently dominant regular cocharacter. We have
\begin{align*}
\CT_{u,v}^{\mu,x,\geq 0}
=
\bigsqcup_{[u',v']_{\tilt}\subseteq [u,v]_{\tilt}}
\CT_{u',v'}^{\mu,x,>0}.
\end{align*}
Each cell $\CT_{u',v'}^{\mu,x,>0}$ is homeomorphic to
$\BR_{>0}^{\ell(u',v')}$. This decomposition is a regular
CW-complex homeomorphic to a closed ball of dimension $\ell(u,v)$.
Its face poset is canonically isomorphic to the poset of subintervals
of $[u,v]_{\tilt}$.
\end{theorem}

\section{Fundamental classes and quantum cohomology}\label{sec:quantum-cohomology}

Throughout this section, we take $k=\BC$.
We use ordinary singular homology and cohomology of the associated
analytic spaces. We write $\cdot$ for the usual cup product and $\star$ for the quantum product.

\subsection{Quantum cohomology and minimal degrees}\label{sec:quantum-minimal-degrees}

For $w\in W_0$, let $\sigma^w\in H^{2\ell(w)}(\CB,\BZ)$ be the
cohomology class of the opposite Schubert variety $\CB^w$.
These classes form a $\BZ$-basis of $H^*(\CB,\BZ)$.
They satisfy $\int_{\CB}\sigma^u\cdot\sigma^{w_0v}=\d_{u,v}$. Let $\o_i\in X^*(T)$, $i\in I$, be the fundamental weights,
characterized by $\<\a_j^\vee,\o_i\>=\d_{ij}$.
For $\l\in X^*(T)$, set
\begin{equation*}
D_\l=\sum_{i\in I}\<\a_i^\vee,\l\>\sigma^{s_i}.
\end{equation*}
In particular, $D_{\o_i}=\sigma^{s_i}$. We identify $H_2(\CB,\BZ)$ with $Q^\vee$ by
$[\CB_{s_i}]\mapsto\a_i^\vee$ for $i\in I$.
The effective degrees form the monoid $Q^\vee_{\geq0}=\sum_{i\in I}\BZ_{\geq0}\a_i^\vee$. Set $\BZ[q]=\BZ[q_i\mid i\in I]$.
Following \cite[\S2.2]{LS10}, for
$\nu=\sum_{i\in I}a_i\a_i^\vee\in Q^\vee_{\geq0}$, we write $q^\nu = \prod_{i\in I}q_i^{a_i}$.

The \emph{small quantum cohomology ring} of $\CB$ is the free
$\BZ[q]$-module $QH^*(\CB)=H^*(\CB,\BZ)\otimes_{\BZ}\BZ[q]$ with product
\begin{equation}\label{eq:qc-product}
\sigma^u\star\sigma^v
= \sum_{w\in W_0, \nu\in Q^\vee_{\geq0}}
c_{u,v}^{w,\nu}\,q^\nu\sigma^w.
\end{equation}
Here $c_{u,v}^{w,\nu}$ is the genus-zero three-point
Gromov--Witten invariant of degree $\nu$ with insertions
$\sigma^u,\sigma^v,\sigma^{w_0w}$.
These invariants are nonnegative integers, and
$c_{u,v}^{w,\nu}=0$ unless $\ell(u)+\ell(v)=\ell(w)+\<\nu,2\rho\>$. The product is associative and commutative, has unit $\sigma^1$,
and specializes to the cup product when $q_i=0$ for every $i\in I$ (cf. \cite[\S2.2]{LS10} and \cite[\S2]{Pos05}). The grading is defined such that $\sigma^w$ has degree $\ell(w)$ and $q^\nu$ has degree $\<\nu,2\rho\>$.

We extend the ordinary cup product $\BZ[q]$-bilinearly
to $QH^*(\CB)$ and interpret integration over $\CB$
coefficient-wise in $q$.
The symmetry of the three-point Gromov--Witten invariants gives
\begin{equation}\label{eq:qc-frobenius}
\int_{\CB}(A_1\star A_2)\cdot A_3
=
\int_{\CB}A_1\cdot(A_2\star A_3)
\end{equation}
for all $A_1,A_2,A_3\in QH^*(\CB)$.

For $A\in QH^*(\CB)$ and $\nu\in Q^\vee_{\geq0}$, we denote by
$[q^\nu]A$ its $H^*(\CB,\BZ)$-valued coefficient of $q^\nu$,
and by $[q^\nu\sigma^w]A$ the $\BZ$-valued coefficient of $q^\nu\sigma^w$.

For $z\in W_0$ and $\l\in X^*(T)$, the quantum Chevalley formula is
\begin{equation}\label{eq:qc-chevalley}
\sigma^z\star D_\l
=
\sum_{z\xrightarrow{\b}z'\text{ in }\QBG(\Phi)}
\<\b^\vee,\l\>\,q^{\wt_{\QBG}(z,z')}\sigma^{z'}
\end{equation}
(cf. \cite[Theorem~2.1]{LS10}).

We now recall the minimal-degree theorem of Postnikov
\cite[Theorem~2 and Corollary~3]{Pos05}.
\begin{theorem}[Postnikov]\label{thm:minimal-quantum-degree}
For any $u,v\in W_0$, every monomial occurring in
$\sigma^u\star\sigma^{w_0v}$ is divisible by $q^{\wt(u,v)}$.
\end{theorem}

\subsection{The evaluation map and the fundamental-class formula}

Choose the maximal compact subgroup $K\subset G$ associated with the
fixed pinning, and put $T_c=K\cap T$. We identify $K/T_c$ with $\CB^{\mathrm{an}}$ naturally. Set
$L_{\mathrm{alg}}K
=\{g\in G(\BC[\e,\e^{-1}])\mid g(S^1)\subset K\}.$
The loop-group Iwasawa decomposition \cite[Theorem~2.1]{LM11} gives a
$T_c$-equivariant homeomorphism
\[
\iota:L_{\mathrm{alg}}K/T_c
\xrightarrow{\sim}\widetilde{\CB}^{\mathrm{an}},
\qquad gT_c\longmapsto g\CI.
\]
Using this realization, define the continuous evaluation map
\[
\operatorname{ev}_1:\widetilde{\CB}^{\mathrm{an}}\longrightarrow K/T_c,
\qquad \operatorname{ev}_1(g\CI)=g(1)T_c
\quad(g\in L_{\mathrm{alg}}K).
\]
Here $T_c$ is regarded as the subgroup of constant loops, so the
formula is independent of the chosen compact-loop representative.

For $x\in W_0$, choose a representative $\dot x\in N_K(T_c)$
and set $R_x(gT_c)=g\dot xT_c$.
This map is well-defined and independent of the representative.
Let $u,v\in W_0$ with $u\leq_xv$, and let $\mu$ be a sufficiently
dominant regular cocharacter.
Recall that $w_{\mu,x}(z)=zt^{-(\mu+\wt(x,z))}x^{-1}$ for $z\in W_0$ (cf. \eqref{eq:psi}). Define
\begin{equation}\label{eq:qc-evaluation}
e_x=
\left.(R_x\circ\operatorname{ev}_1)\right|_{(\CT_{u,v}^{\mu,x})^{\mathrm{an}}}:
(\CT_{u,v}^{\mu,x})^{\mathrm{an}}
\longrightarrow\CB^{\mathrm{an}}.
\end{equation}
For a cocharacter $\l\in Q^\vee$, the loop representing $t^\l$ sends $1 \in S^1$ to the identity of $T$. Thus, $e_x\bigl(\dot{w}_{\mu,x}(z)\CI\bigr)=\dot zT_c$ for $z\in[u,v]_{\tilt}$ (cf. \cite[Lemma~3.6]{LLS22}).

The main result of this section is the following.
\begin{theorem}\label{thm:quantum-class}
Let $u,v\in W_0$. Let $x \in W_0$ be such that $u\leq_x v$ and let $\mu$ be a sufficiently
dominant regular cocharacter.
Then the following equality holds in
$H_{2\ell(u,v)}(\CB,\BZ)$:
\begin{equation}\label{eq:qc-class}
(e_x)_*[\CT_{u,v}^{\mu,x}]
=
\operatorname{PD}\left(
[q^{\wt(u,v)}]
\bigl(\sigma^u\star\sigma^{w_0v}\bigr)
\right).
\end{equation}
Here, we write $\operatorname{PD}$ for Poincar\'e duality
from cohomology to homology on $\CB$.
In particular, the class $(e_x)_*[\CT_{u,v}^{\mu,x}]$ is independent of $\mu$ and $x$.
\end{theorem}

We first recall the notation needed for the proof.

Put $\fkt=\Lie(T)$ and
$\widehat{\fkt}=\fkt\oplus\BC c \oplus \BC d_{\mathrm{rot}}$,
where $c=\ta_0^\vee+\th^\vee$ and
$d_{\mathrm{rot}}$ is the loop derivation. The affine fundamental weights $\L_i$, $i\in\tilde I$, are
defined by
$\<\ta_j^\vee,\L_i\>=\d_{ij}$ and
$\<d_{\mathrm{rot}},\L_i\>=0$ for $i,j\in\tilde I$.
For $\l\in X^*(T)$, let $\operatorname{af}(\l)$ be its extension
to $\widehat{\fkt}$ which vanishes on $c$ and $d_{\mathrm{rot}}$.
Since $\<c,\L_0\>=1$, we have
$\operatorname{af}(\o_i)=\L_i-\<c,\L_i\>\L_0$ for $i\in I$.

For $y\in W_{\aff}$, write
$\widetilde{\CB}_y=\overline{\CI\dot y\CI/\CI}$ for the affine
Schubert variety, and let $\widetilde\sigma^y$ be the affine
Schubert cohomology class dual to $[\widetilde{\CB}_y]$.
For an affine integral weight $\eta$, denote by
$\widetilde D_\eta$ its divisor class, normalized by
$\widetilde D_{\L_i}=\widetilde\sigma^{s_i}$.
Thus
$\widetilde D_\eta
=\sum_{i\in\tilde I}\<\ta_i^\vee,\eta\>\widetilde\sigma^{s_i}$.
Set $D_\l^{\aff}=\widetilde D_{\operatorname{af}(\l)}$.

By \cite[Lemma~3.7 and \S5.3]{LLS22}, applied to first Chern
classes, $\operatorname{ev}_1^*D_\l=D_\l^{\aff}$.
Moreover, the identity
$gt\dot x=g\dot x(\dot x^{-1}t\dot x)$ for $t\in T_c$
shows that $R_x^*D_\l=D_{x(\l)}$.
Consequently,
\begin{equation}\label{eq:qc-divisor-pullback}
e_x^*D_\l
=
\left.D_{x(\l)}^{\aff}\right|_{\CT_{u,v}^{\mu,x}}.
\end{equation}

Let $z\xrightarrow{\b}zs_\b$ be an edge in $\QBG(\Phi)$ such that $z\le_x zs_\b$. By Proposition \ref{prop:path-corr}, it corresponds to an upward Bruhat edge $w_{\mu,x}(z) \xrightarrow{\tg} w_{\mu,x}(zs_{\b})$, where $\tg=x(\b)+\left(
\d_{\Phi^-}(z(\b))
+\<\mu+\wt(x,z),\b\>
\right)\d\in\Phi_{\aff}^+$. In this case, we have
\begin{equation}\label{eq:qc-root-pairing}
\<\tg^\vee,\operatorname{af}(x(\l))\>
=
\<x(\b^\vee),x(\l)\>
=
\<\b^\vee,\l\>.
\end{equation}

\begin{proof}[Proof of Theorem \ref{thm:quantum-class}]
Note that $H^*(\CB,\BQ)$ is generated by divisor classes and
$H_*(\CB,\BZ)$ is torsion-free. Let $\l_1,\ldots,\l_{\ell(u,v)}\in X^*(T)$, and set
$F=D_{\l_1}\cdot\ldots\cdot D_{\l_{\ell(u,v)}}$. It suffices to prove 
\begin{equation}\label{eq:anyF}
\int_{\CT_{u,v}^{\mu,x}}e_x^*F = \int_{\CB}F\cdot
  [q^{\wt(u,v)}](\sigma^u\star\sigma^{w_0v}).
\end{equation}

Set $P=D_{\l_1}\star\ldots\star D_{\l_{\ell(u,v)}}$ and write $P=F+\sum_{\xi>0}q^\xi F_\xi$,
where $F_\xi\in H^*(\CB,\BZ)$ and
$\xi>0$ means $\xi\in Q^\vee_{\geq0}\setminus\{0\}$.
By Theorem~\ref{thm:minimal-quantum-degree}, every degree occurring
in $\sigma^u\star\sigma^{w_0v}$ is at least $\wt(u,v)$.
Thus the terms $q^\xi F_\xi$ with $\xi>0$ do not contribute
to the coefficient of $q^{\wt(u,v)}$ below.
Using \eqref{eq:qc-frobenius}, we obtain
\begin{align*}
\int_{\CB}F\cdot
 [q^{\wt(u,v)}](\sigma^u\star\sigma^{w_0v})
&=[q^{\wt(u,v)}]\int_{\CB}
 P\cdot(\sigma^u\star\sigma^{w_0v})\\
&=[q^{\wt(u,v)}]\int_{\CB}
 (\sigma^u\star P)\cdot\sigma^{w_0v}\\
&=[q^{\wt(u,v)}\sigma^v](\sigma^u\star P).
\end{align*}
Hence, it suffices to prove that
\begin{equation}\label{eq:anyF2}
\int_{\CT_{u,v}^{\mu,x}}e_x^*F = [q^{\wt(u,v)}\sigma^v](\sigma^u\star P).
\end{equation}

By \eqref{eq:qc-divisor-pullback} and the affine
Chevalley formula, we have
\[
\int_{\CT_{u,v}^{\mu,x}}e_x^*F
=
\sum_{\widetilde p}
\prod_{j=1}^{\ell(u,v)}
\<\tg_j^\vee,\operatorname{af}(x(\l_j))\>,
\]
where $\widetilde p$ runs over all maximal Bruhat chains $w_{\mu,x}(u)=y_0
\xrightarrow{\tg_1}y_1
\xrightarrow{\tg_2}\cdots
\xrightarrow{\tg_{\ell(u,v)}}y_{\ell(u,v)}
=w_{\mu,x}(v)$.

On the other hand, by the quantum Chevalley formula \eqref{eq:qc-chevalley}, we get
\[
[q^{\wt(u,v)}\sigma^v](\sigma^u\star P)
=
\sum_p
\prod_{j=1}^{\ell(u,v)}\<\b_j^\vee,\l_j\>,
\]
where $p$ runs over all shortest paths $u
\xrightarrow{\b_1}z_1
\xrightarrow{\b_2}\cdots
\xrightarrow{\b_{\ell(u,v)}}z_{\ell(u,v)}=v$
in $\QBG(\Phi)$.

By \eqref{eq:qc-root-pairing}, we have $\<\tg_j^\vee,\operatorname{af}(x(\l_j))\>
=\<\b_j^\vee,\l_j\>$ for every $j$. By Proposition~\ref{prop:path-corr}, the chains $\widetilde p$ and the paths $p$ are in one-to-one correspondence. Thus the two sums agree, proving
\eqref{eq:anyF2}.
\end{proof}

\subsection{Realizing tilted Richardson varieties in $G/B$}
\label{sec:realization-in-flag}

Let $u,v\in W_0$. Choose $x\in W_0$ such that $u\leq_xv$
and a sufficiently dominant regular cocharacter $\mu$.
Define
$$\CT_{u,v}^{\mu,x,\CB} = e_x(\CT_{u,v}^{\mu,x})\subseteq\CB.$$
Note that this image is compact in the complex analytic
topology. Theorem \ref{thm:quantum-class} shows that
$(e_x)_*[\CT_{u,v}^{\mu,x}]$ is independent of $\mu$ and $x$.

\begin{conjecture}\label{conj:realization-in-flag}
\begin{enumerate}
\item The subset $\CT_{u,v}^{\mu,x,\CB}$ is
independent of the choices of $\mu$ and $x$. We denote the common image by $\CT_{u,v}^{\CB}\subseteq\CB$.

\item $\CT_{u,v}^{\CB}\subseteq\CB$ is a closed algebraic subvariety of $\CB$.

\item In the case of type $A$, $\CT_{u,v}^{\CB}\subseteq\CB$ coincides with the tilted Richardson variety $\CT_{u,v}^{\mathrm{GGG}}$ defined in \cite{GGG26}.

\end{enumerate}
\end{conjecture}


\begin{thebibliography}{99}

\bibitem{BH21}
H.~Bao and X.~He.
\newblock {Flag manifolds over semifields}.
\newblock {\em Algebra Number Theory}, 15(8):2037--2069, 2021.

\bibitem{BH24}
H.~Bao and X.~He.
\newblock {Product structure and regularity theorem for totally nonnegative
flag varieties}.
\newblock {\em Invent. Math.}, 237(1):1--47, 2024.

\bibitem{Bjor84}
A.~Bj{\"o}rner.
\newblock {Posets, regular {CW} complexes and {Bruhat} order}.
\newblock {\em European J. Combin.}, 5(1):7--16, 1984.

\bibitem{BFP}
F.~Brenti, S.~Fomin, and A.~Postnikov.
\newblock {Mixed {Bruhat} operators and {Yang--Baxter} equations
for {Weyl} groups}.
\newblock {\em Int. Math. Res. Not. IMRN}, 1999(8):419--441, 1999.

\bibitem{Deod85}
V.~V.~Deodhar.
\newblock {On some geometric aspects of {Bruhat} orderings.
{I}. {A} finer decomposition of {Bruhat} cells}.
\newblock {\em Invent. Math.}, 79(3):499--511, 1985.

\bibitem{Dyer93}
M.~J.~Dyer.
\newblock {Hecke algebras and shellings of {Bruhat} intervals}.
\newblock {\em Compositio Math.}, 89(1):91--115, 1993.

\bibitem{FGP97}
S.~Fomin, S.~Gelfand, and A.~Postnikov.
\newblock {Quantum {Schubert} polynomials}.
\newblock {\em J. Amer. Math. Soc.}, 10(3):565--596, 1997.

\bibitem{GKL22}
P.~Galashin, S.~N.~Karp, and T.~Lam.
\newblock {Regularity theorem for totally nonnegative flag varieties}.
\newblock {\em J. Amer. Math. Soc.}, 35(2):513--579, 2022.

\bibitem{GGG26}
J.~Gao, S.~Gao, and Y.~Gao.
\newblock {Tilted {Richardson} varieties}.
\newblock Preprint, 2026.
\newblock \href{https://arxiv.org/abs/2602.05326}{arXiv:2602.05326}.

\bibitem{HR76}
M.~Hochster and J.~L.~Roberts.
\newblock {The purity of the {Frobenius} and local cohomology}.
\newblock {\em Adv. Math.}, 21(2):117--172, 1976.

\bibitem{KL79}
D.~Kazhdan and G.~Lusztig.
\newblock {Representations of {Coxeter} groups and {Hecke} algebras}.
\newblock {\em Invent. Math.}, 53(2):165--184, 1979.

\bibitem{Kum17}
S.~Kumar.
\newblock {Positivity in {$T$}-equivariant {$K$}-theory of flag varieties
associated to {Kac--Moody} groups
(with an appendix by M.~Kashiwara)}.
\newblock {\em J. Eur. Math. Soc.}, 19(8):2469--2519, 2017.

\bibitem{KS14}
S.~Kumar and K.~Schwede.
\newblock {{Richardson} varieties have {Kawamata} log terminal singularities}.
\newblock {\em Int. Math. Res. Not. IMRN}, 2014(3):842--864, 2014.

\bibitem{LLS22}
T.~Lam, S.~J.~Lee, and M.~Shimozono.
\newblock {On the coproduct in affine {Schubert} calculus}.
\newblock In {\em Facets of Algebraic Geometry}, Vol.~2,
volume 473 of {\em London Math. Soc. Lecture Note Ser.},
pages 115--146.
Cambridge University Press, Cambridge, 2022.

\bibitem{LS10}
T.~Lam and M.~Shimozono.
\newblock {Quantum cohomology of {$G/P$} and homology of affine
{Grassmannian}}.
\newblock {\em Acta Math.}, 204(1):49--90, 2010.

\bibitem{LM11}
P.~J.~Littig and S.~A.~Mitchell.
\newblock {Generating varieties for affine {Grassmannians}}.
\newblock {\em Trans. Amer. Math. Soc.}, 363(7):3717--3731, 2011.

\bibitem{Lusz94}
G.~Lusztig.
\newblock {Total positivity in reductive groups}.
\newblock In {\em Lie Theory and Geometry},
volume 123 of {\em Progr. Math.}, pages 531--568.
Birkh\"auser Boston, Boston, MA, 1994.

\bibitem{MR04}
B.~R.~Marsh and K.~Rietsch.
\newblock {Parametrizations of flag varieties}.
\newblock {\em Represent. Theory}, 8:212--242, 2004.

\bibitem{MV20}
E.~Mili\'{c}evi\'{c} and E.~Viehmann.
\newblock {Generic {Newton} points and the {Newton} poset
in {Iwahori}-double cosets}.
\newblock {\em Forum Math. Sigma}, 8, Paper No.~e50, 18~pp., 2020.

\bibitem{NW17}
S.~Naito and H.~Watanabe.
\newblock {A combinatorial formula expressing periodic {$R$}-polynomials}.
\newblock {\em J. Combin. Theory Ser. A}, 148:197--243, 2017.

\bibitem{PR08}
G.~Pappas and M.~Rapoport.
\newblock {Twisted loop groups and their affine flag varieties}.
\newblock {\em Adv. Math.}, 219(1):118--198, 2008.

\bibitem{Pos05}
A.~Postnikov.
\newblock {Quantum {Bruhat} graph and {Schubert} polynomials}.
\newblock {\em Proc. Amer. Math. Soc.}, 133(3):699--709, 2005.

\bibitem{Ric92}
R.~W.~Richardson.
\newblock {Intersections of double cosets in algebraic groups}.
\newblock {\em Indag. Math. (N.S.)}, 3(1):69--77, 1992.

\bibitem{Riet99}
K.~Rietsch.
\newblock {An algebraic cell decomposition of the nonnegative part
of a flag variety}.
\newblock {\em J. Algebra}, 213(1):144--154, 1999.

\bibitem{Riet06}
K.~Rietsch.
\newblock {Closure relations for totally nonnegative cells in {$G/P$}}.
\newblock {\em Math. Res. Lett.}, 13(5--6):775--786, 2006.

\bibitem{Sch24}
F.~Schremmer.
\newblock {Affine {Bruhat} order and {Demazure} products}.
\newblock {\em Forum Math. Sigma}, 12, Paper No.~e53, 56~pp., 2024.

\bibitem{Sch25a}
F.~Schremmer.
\newblock {Affine {Deligne--Lusztig} varieties via the double {Bruhat}
graph, {I}: {Semi-infinite} orbits}.
\newblock {\em Algebra Number Theory}, 19(10):1973--2014, 2025.

\bibitem{Sch23}
F.~Schremmer.
\newblock {Generic {Newton} points and cordial elements}.
\newblock {\em Adv. Math.}, 502, Part A, Paper No.~111120, 2026.

\bibitem{LW}
L.~K.~Williams.
\newblock {Shelling totally nonnegative flag varieties}.
\newblock {\em J. Reine Angew. Math.}, 609:1--21, 2007.

\end{thebibliography}
\end{document}